\documentclass{amsart}[12pt]
\usepackage{amsmath,amsfonts,amssymb,latexsym,amsthm}
\usepackage{epsfig}
\usepackage{hyperref} 
\usepackage[dvipsnames]{xcolor}
\usepackage{tikz}

\numberwithin{equation}{section}

\usepackage{tikz-cd}
\usepackage{color,soul}
\usepackage{array}
\usepackage{zref-savepos}
\usepackage{multirow, bigstrut}
\usepackage[font={small,it}]{caption}
\usepackage{tikz}
\usepackage{bm}

\usetikzlibrary{arrows}

\usepackage{hyperref}
\hypersetup{
    colorlinks=true,
    linkcolor=blue,
    filecolor=magenta,
    urlcolor=cyan,
}

\usetikzlibrary{calc}
\usepackage{comment}
\usepackage{enumerate}
\usepackage{tikz}
\usetikzlibrary{shapes.geometric}

\newtheorem{theorem}{Theorem}[section]
\newtheorem*{theorem*}{Theorem}
\newtheorem{lemma}[theorem]{Lemma}
\newtheorem{proposition}[theorem]{Proposition}
\newtheorem{corollary}[theorem]{Corollary}

\newtheorem{example}[theorem]{Example}

\theoremstyle{definition}
\newtheorem{definition}[theorem]{Definition}
\newtheorem{remark}[theorem]{Remark}

\theoremstyle{plain}

\newcommand{\C}{\mathbb{C}}
\newcommand{\R}{\mathbb{R}}
\newcommand{\N}{\mathbb{N}}
\newcommand{\Z}{\mathbb{Z}}

\newcommand{\He}{\mathcal{H}}

\DeclareMathOperator{\cpc}{Cap}

\DeclareMathOperator{\BCT}{BCT}
\DeclareMathOperator{\CT}{CT}
\DeclareMathOperator{\CTI}{I}
\DeclareMathOperator{\Vol}{vol}

\def\sq{\square}
\def\.{\hskip.06cm}
\def\ts{\hskip.03cm}
\def\wt{\widetilde}

\begin{document}

\title[Lower bounds for contingency tables]{Lower bounds for contingency tables via Lorentzian polynomials}
\author[Petter Br\"and\'en, Jonathan Leake, and Igor Pak]{Petter Br\"and\'en$^\star$, Jonathan Leake$^\star$, and Igor Pak$^\diamond$}
\date{\today}

\thanks{\thinspace ${\hspace{-1.5ex}}^\star$Department of Mathematics,
KTH Royal Institute of Technology, Stockholm, Sweden.
\hskip.06cm
Email:
\hskip.06cm \texttt{pbranden@kth.se} and \texttt{leake@kth.se}}

\thanks{\thinspace ${\hspace{-1.5ex}}^\diamond$Department of Mathematics,
UCLA, Los Angeles, CA, 90095.
\hskip.06cm
Email:
\hskip.06cm
\texttt{pak@math.ucla.edu}}

\begin{abstract}
We present a new lower bound on the number of \emph{contingency tables},
improving upon and extending previous lower bounds by Barvinok~\cite{Bar09,Bar16}
and Gurvits~\cite{Gur15}. As an application, we obtain new lower bounds
on the volumes of \emph{flow} and \emph{transportation polytopes}.
Our proofs are based on recent results on \emph{Lorentzian polynomials}.
\end{abstract}

\maketitle

\section{Introduction} \label{sec:intro}

\emph{Contingency tables} are fundamental objects across the sciences.  In statistics,
they are employed to study dependence structure between two or more variables,
see e.g.~\cite{Ev,FLL,Kat}.  They  play an important role in combinatorics
and graph theory since they are in bijection with bipartite multi-graphs
with given degrees, see e.g.~\cite{Bar09,DG95}. In discrete geometry
and combinatorial optimization, they are frequently studied
as integer points in transportation polytopes~\cite{DK}.  They also
appear in a variety of other contexts, from algebraic and enumerative
combinatorics~\cite{Pak,PP} to commutative algebra~\cite{DS} and
topology~\cite{KS}.

Motivated by these connections and applications, a great deal of effort is made
to approximate and to estimate the number of contingency tables, both
theoretically and computationally.  For that, a variety of tools
have been developed in different areas, such as the traditional
and probabilistic divide-and-conquer \cite{DZ,GaM}, the asymptotic
analysis \cite{BH12,CM10,GM}, the MCMC algorithms \cite{CDGJM,DKM},
approximation algorithms~\cite{AH,BLSY}, and integer programming
\cite{BDV,D09}.

In this paper we present a new lower bound (Theorem~\ref{thm:main_general_bound})
on the number of \emph{contingency tables with cell-bounded entries}, a setting which
includes a \emph{permanent}.  This bound improves upon previous lower bounds,
holds for all marginals, is fast to compute, and behaves well in many examples.
We begin with an important special case.

\smallskip

Let \ts $\bm{\alpha} = (\alpha_1,\ldots,\alpha_m) \in \N^m$ \ts and
\ts $\bm{\beta} = (\beta_1,\ldots,\beta_n) \in \N^n$ \ts be integer vectors.
A \emph{contingency table} with marginals \ts $(\bm{\alpha},\bm{\beta})$ \ts
is an $m \times n$ matrix $A=(a_{ij})$, such that \ts $a_{ij}\in \N$,
\[
    \sum_{i=1}^m \. a_{ij} \. = \. \beta_j \text{\, for all \, $1\le j\le n\,,$}
    \quad \text{and} \quad \sum_{j=1}^n \. a_{ij} \. =\. \alpha_i \text{\, for all \, $1\le i\le m\ts.$}
\]
Denote by \ts $\CT(\bm\alpha,\bm\beta)$ \ts the number of contingency tables
with marginals \ts $(\bm\alpha,\bm\beta)$.

\medskip

\begin{theorem}[{= Corollary~\ref{c:main-CT}}] \label{t:main-intro}
For every \ts $\bm{\alpha}, \bm{\beta}$ \ts as above, we have:
$$
\cpc_{\bm\alpha\,\bm\beta} \, \geq \, \CT(\bm\alpha,\bm\beta) \, \geq \,
\Biggl[\frac{1}{e^{m+n-1}} \, \prod_{i=2}^m \.\frac{1}{\alpha_i+1} \, \prod_{j=1}^n \.\frac{1}{\beta_j+1}\Biggr] \,
\cpc_{\bm\alpha\,\bm\beta}\,,
$$
where
$$\cpc_{\bm\alpha\,\bm\beta} \, := \, \. \inf_{\substack{x_i \in (0,1) \\ 1\le i\le m}} \,\.
 \inf_{\substack{y_j \in (0,1) \\ 1\le j\le n}} \ \Biggl[\.\prod_{i=1}^{m} x_i^{\alpha_i} \,
 \prod_{j=1}^{m} y_j^{\beta_j}\Biggr]^{-1} \, \prod_{i=1}^{m}\. \prod_{j=1}^{n} \. \frac{1}{1-x_iy_j}\,.
$$
\end{theorem}

\medskip

Here the upper bound is elementary and follows from the definition of contingency tables.
The lower bound is a special case of our Main Theorem~\ref{thm:main_general_bound},
which is an improvement over Barvinok's lower bound~\cite{Bar09}
(Theorem~\ref{thm:barvinok_general_bound}). Note that Main Theorem~\ref{thm:main_general_bound}
gives a stronger bound when $\alpha_i, \beta_j$ are bounded, and generalizes to the
case of cell-bounded contingency tables, i.e.\ to all \ts $A=(a_{ij})$, such that \ \ts $a_{ij}\le k_{ij}$.
When all \ts $k_{ij}\in \{0,1\}$, the corresponding
contingency tables are in bijection with perfect matchings in a bipartite graph with
adjacency matrix \ts $K=(k_{ij})$, and our results recover Gurvits's recent lower
bound in~\cite{Gur15} (Theorem~\ref{thm:main_binary_bound}).


There are several ways to understand the theorem.  First, it is a
general theoretical result in line with other results on contingency tables
and graph counting, and can be used in both asymptotic enumeration and
analysis of graph algorithms (see e.g.~\cite{Bar16,Wor} and references therein).
Second, the bound in the theorem can be effectively computed,
and thus gives a fast approximation algorithm for a problem of computing
\ts $\CT(\bm\alpha,\bm\beta)$, see~$\S$\ref{ss:finrem-alg}.

Third, letting $\alpha_i, \beta_j \to \infty$ at the same rate in
Theorem~\ref{t:main-intro}, allows us to obtain new lower bounds
for the volumes of \emph{transportation polytopes}.  We explore this
connection in Section~\ref{sec:volume}.
Finally, the theorem can be used to obtain exact bounds
in many special cases of interest in the statistics literature,
and then make explicit comparisons with other bounds;
we report our numerical experiments in Section~\ref{sec:numerical}.

\medskip

\noindent
{\bf Paper structure.} \ts We begin with
a lengthy Section~\ref{sec:main_results} stating our new results
on the number of contingency tables and
discussing prior work on the subject.  We continue presenting new
results in Section~\ref{sec:random}, this time in probabilistic
framework of contingency tables with random constraints.
In Sections~\ref{sec:Lorenz}
and~\ref{sec:capacity} we discuss some known results on Lorentzian
and related classes of
polynomials, and then derive new capacity bounds for coefficients.  In
Section~\ref{sec:proofs} we prove the main results of the paper.

In the second part of the paper, we give applications of our results
and explore connections to earlier work.  First, in
Section~\ref{sec:compare_bounds}, we prove that our bounds
are sharper than the two earlier bounds by Barvinok, at least
asymptotically.  In the next two Sections~\ref{sec:volume}
and~\ref{sec:uniform}, we apply our results to computing lower bounds
for the volume of flow and transportation polytopes,
and to the special case of uniform marginals.
We conclude with numerical examples in Section~\ref{sec:numerical},
and final remarks in Section~\ref{sec:finrem}.

We should mention that the paper is very far from being
self-contained, as we repeatedly use available tools with very
little preparation, but with the exact references to the literature.
This is why we upfronted the results in Sections~\ref{sec:main_results}
and~\ref{sec:random}, to ease access to our theorems.  We repeat
the pattern in Sections~\ref{sec:volume} and~\ref{sec:uniform}.

\bigskip

\section{Main results and prior work} \label{sec:main_results}

The main results in this paper are lower bounds on $\CT_K(\bm\alpha,\bm\beta)$
for general marginals \ts $(\bm\alpha,\bm\beta)$, and cell-bounded entries given by
matrix~$K$.  In this section we present the results and compare them with
similar lower bounds due to Barvinok and Gurvits.  Along with the actual bounds,
we also give asymptotic bounds by naively applying Stirling's
approximation wherever possible. Since Stirling's approximation
is off by at most a factor of \ts $\frac{e}{\sqrt{2\pi}}$,
the asymptotic values given are decent approximations for all asymptotic regimes.
For two multivariate functions $F(\bm{a})$ and $G(\bm{a})$,
$\bm{a}=(a_1,\ldots,a_n) \in \N^n$, we write $F\gtrsim G$ if
$F(\bm{a}) \ge C\cdot G(\bm{a})$ for a universal constant $C>0$
and $\min\{a_i\}\to \infty$.

\medskip

\subsection{Definitions}\label{ss:main-basic}
Let \ts $\bm{\alpha} = (\alpha_1,\ldots,\alpha_m) \in \N^m$ \ts and \ts
$\bm{\beta} = (\beta_1,\ldots,\beta_n) \in \N^n$ \ts be integer vectors
of \emph{marginals}, such that
$$\sum_{i=1}^m \ts \alpha_i \, = \, \sum_{j=1}^n \ts \beta_j \, = \. N\ts.
$$
Let $K=(k_{ij})$ be an $m\times n$ matrix with entries in $\N \cup \{+\infty\}$.
We use $K=\infty$ to denote an all $\infty$ matrix, and $K=\bm{1}$ to denote
the all-one matrix.

Matrix $A=(a_{ij})$ is called a \emph{$K$-contingency table}
with marginals $(\bm{\alpha},\bm{\beta})$, if $A$ is a contingency table
with cell-bounded entries \ts $0\le a_{ij}\le k_{ij}$.
When $K=\infty$ we obtain the usual (unrestricted) contingency tables.
When $K=\bm{1}$, matrix $A$ is called a \emph{binary contingency table}.

As in the introduction, let \ts $\CT(\bm\alpha,\bm\beta)$ \ts denote the
number of contingency tables with marginals $(\bm\alpha,\bm\beta)$, and let
\ts $\CT_K(\bm\alpha,\bm\beta)$ \ts denote the number of $K$-contingency
tables with the same marginals.  When all $k_{ij}\in \{0,1\}$, we
call such matrix $K=(k_{ij})$ \emph{graphical}, and write \ts $\BCT_K(\alpha,\beta)$ \ts
for the number of binary contingency tables.  In particular,
when all $k_{ij}=1$, we may further write \ts $\BCT(\alpha,\beta)$ \ts
for the number of such tables. Finally, when all  $k_{ij}\in \{0,+\infty\}$,
we call such $K$ \emph{multigraphical}.

Consider the generating polynomial \ts  $P_K(\bm{x},\bm{y}) \in \N[\bm{x},\bm{y}]$ \ts
for all contingency tables w.r.t.\ their marginals:
$$
P_K(\bm{x},\bm{y}) \, := \, \prod_{i=1}^m \. \prod_{j=1}^n \, \sum_{\ell=0}^{k_{ij}} \. x_i^\ell \ts y_j^\ell
\, = \, \sum_{\bm{\alpha}, \. \bm{\beta}} \. \bm{x}^{\bm{\alpha}} \ts \bm{y}^{\bm{\beta}} \ts
\CT_K(\bm{\alpha},\bm{\beta})\ts,
$$
where \ts $\bm{x}^{\bm{\alpha}} \bm{y}^{\bm{\beta}} \ts = \ts
x_1^{\alpha_1} \ts \cdots \ts x_m^{\alpha_m} \ts y_1^{\beta_1}\ts \cdots \ts y_n^{\beta_n}$.
In a special case, for (usual) and binary contingency tables the generating polynomials are given by
$$
P_\infty(\bm{x}, \bm{y}) \, := \, \prod_{i=1}^m \. \prod_{j=1}^n \. \frac{1}{1-x_iy_j} \quad
\text{and} \quad P_{\bf 1}(\bm{x}, \bm{y}) \, := \, \prod_{i=1}^m \. \prod_{j=1}^n \. (1+x_iy_j)\..
$$

\noindent
Let \ts $F\in \R_+[\bm{x},\bm{y}]$ be a polynomial in $(m+n)$ variables.
Define the \emph{capacity} of $F$ as
$$\cpc_{\bm\alpha\,\bm\beta}(F) \, := \, \. \inf_{\bm{x},\ts\bm{y} \ts > \ts 0}  \.
\bm{x}^{-\bm{\alpha}} \ts \bm{y}^{-\bm{\beta}}  \ts  F(\bm{x},\bm{y})\,.
$$
We consider capacity of polynomials $P_K$ and other polynomials throughout the paper.

\medskip

\subsection{Main theorem}\label{ss:main-thm}
We are now ready to state the main result of the paper.

\smallskip

\begin{theorem}[Main theorem] \label{thm:main_general_bound}
Let \ts $\bm\alpha \in \N^m$, \ts $\bm\beta \in \N^n$,
such that  $\sum_i \alpha_i = \sum_j \beta_j$.  Let \ts
$K=(k_{ij})$ be an $m \times n$ matrix with all entries \ts
$k_{ij}\in\N \cup \{+\infty\}$.  Then:
\[
\cpc_{\bm\alpha\,\bm\beta}(P_K) \, \geq \, \CT_K(\bm\alpha,\bm\beta) \, \geq \,
\Biggl[\prod_{i=2}^m \. \frac{a_i^{a_i}}{(a_i+1)^{a_i+1}} \, \prod_{j=1}^n \. \frac{b_j^{b_j}}{(b_j+1)^{b_j+1}}\Biggr]
\. \cpc_{\bm\alpha\,\bm\beta}(P_K)\ts,
\]
where \ts $a_i := \min\{\alpha_i, \lambda_i-\alpha_i\}$,
\ts $b_i := \min\{\beta_j, \gamma_j-\beta_j\}$, \ts
$\lambda_i := \sum_j k_{ij}$ \ts and \ts $\gamma_j := \sum_i k_{ij}$\ts,
for all \ts $1\le i \le m$, \ts $1 \le j \le n$.
\end{theorem}

\smallskip

Note that the first product starting with $i=2$ is not a typo,
but a feature of the bound.  Given a choice, the lower bound is
the largest when $a_1$ is chosen to be maximal.  Let us mention that
the capacity constant in the theorem can be computed in polynomial
time (see~$\S$\ref{ss:finrem-alg}).

Before moving on, let us emphasize the unrestricted nature of the matrix~$K$.
In previously known bounds, entries in $K$ were either restricted, or
certain values of $K$ were allowed by weighting variables in certain ways.
For us though, the value of $K$ is inconsequential to our proof method,
and so it shows up as a parameter in our bounds in a straightforward way.

\medskip

\subsection{Two special cases}\label{ss:main-two}
First, in the unrestricted case \ts $K=\infty$, we obtain the following result:

\smallskip

\begin{corollary}[{= Theorem~\ref{t:main-intro}}] \label{c:main-CT}
Let \ts $\bm\alpha \in \N^m$, \ts $\bm\beta \in \N^n$, such that
$\sum_i \alpha_i = \sum_j \beta_j =N$.  Then:
    \[
        \cpc_{\bm\alpha\,\bm\beta} \, \geq \,\CT(\bm\alpha,\bm\beta) \,
        \geq \, \Biggl[\frac{1}{e^{m+n-1}} \prod_{i=2}^m \frac{1}{\alpha_i+1} \prod_{j=1}^n \frac{1}{\beta_j+1} \Biggr]
        \. \cpc_{\bm\alpha\,\bm\beta}\..
    \]
    In particular, we have
    \[
        \cpc_{\bm\alpha\,\bm\beta} \, \geq \, \CT(\bm\alpha,\bm\beta) \,
        \geq \, e^{-4N} \cdot \cpc_{\bm\alpha\,\bm\beta}\..
    \]
\end{corollary}

\smallskip

Second, in the case that $\alpha_i, \ts \beta_j$ are bounded for $i,j \geq 2$,
the approximation ratio we obtain is independent of the dominant marginal values,
$\alpha_1$ and $\beta_1$.  Note that the number of terms in the polynomial $P_K$
in the following result \emph{does} depend on the dominant marginal values,
even though the approximation ratio does not.

\smallskip

\begin{theorem} \label{thm:almost_const}
Let \ts $\bm\alpha \in \N^m$, \ts $\bm\beta \in \N^n$, such that
$\sum_i \alpha_i = \sum_j \beta_j$, and \ts $\alpha_i,\beta_j \leq c$ \ts for all $i,j \geq 2$.
Then:
    \[
        \cpc_{\bm\alpha\,\bm\beta}(P_K) \, \geq \, \CT(\bm\alpha,\bm\beta)
        \, \geq \, \frac{1}{(m+n-1) \ts (e(c+1))^{m+n-1}} \.\cdot\. \cpc_{\bm\alpha\,\bm\beta}(P_K)
    \]
where \ts $K=(k_{ij})$, and \ts $k_{ij} = \min(\alpha_i,\beta_j)$ \ts for all~$i,j$.
Note that this is the entrywise minimal value of $K$ such that \ts
$\CT(\bm\alpha,\bm\beta) = \CT_K(\bm\alpha,\bm\beta)$.
\end{theorem}

\medskip

\subsection{Barvinok's first bound}\label{ss:main-barv-1}
For the case of general contingency tables with a multigraphical matrix
$K=(k_{ij})$, Barvinok gives the following bounds:

\begin{theorem}[{Barvinok~\cite[Thm~1.3]{Bar09}}] \label{thm:barvinok_general_bound}
Let \ts $\bm\alpha \in \N^m$, \ts $\bm\beta \in \N^n$, such that \ \ts $N = \sum_i \alpha_i = \sum_j \beta_j$.
Let $K=(k_{ij})$  be an $m \times n$ multigraphical
matrix, i.e.\ \ts $k_{ij}\in \{0,+\infty\}$, for all \ts  $1\le i\le m$ \ts
and \ts $1\le j \le n$. Then, for $m+n \geq 10$, we have:
    \[
        \cpc_{\bm\alpha\,\bm\beta}(P_K) \, \geq \, \CT_K(\bm\alpha,\bm\beta)
        \, \geq \, C_\mathrm{Barv}(K,\bm\alpha,\bm\beta) \, \cpc_{\bm\alpha\,\bm\beta}(P_K)\ts,
    \]
where
    \[
        \begin{split}
            C_\mathrm{Barv}(K,\bm\alpha,\bm\beta) \, &= \,\.
            \frac{\Gamma(\frac{m+n}{2})}{2\ts e^5 \ts \pi^{\frac{m+n-2}{2}} \ts mn \ts (N+mn)} \ts
            \left(\frac{2}{(mn)^2(N+1)(N+mn)}\right)^{m+n-1} \\
                &\qquad \times \. \frac{N!\ts (N+mn)! \ts (mn)^{mn}}{N^N  (N+mn)^{N+mn} \ts (mn)!} \, \.
                \prod_{i=1}^m \. \frac{\alpha_i^{\alpha_i}}{\alpha_i!} \,  \prod_{j=1}^n \. \frac{\beta_j^{\beta_j}}{\beta_j!}\..
        \end{split}
    \]
\end{theorem}

We also give a simplified version of Barvinok's bound in the following.

\begin{theorem}\label{t:barvinok-CT-const}
    The value of $C_\mathrm{Barv}(K,\bm\alpha,\bm\beta)$ in the previous theorem is such that asymptotically we have
    \[
        R^{n+m-1} \ts S \, \gtrsim \, C_\mathrm{Barv}(K,\bm\alpha,\bm\beta) \gtrsim R^{n+m} \ts S,
    \]
    where \ts $\min\{\alpha_i,\beta_j\}\to \infty$, \ts $\min\{m,n\}\to \infty$,
    $$
    R \, = \, \frac{(m+n)^{\frac{1}{2}}}{\pi \ts e^{\frac{1}{2}} \ts (mn)^2 \ts (N+1)\ts (N+mn)}, \qquad \text{and} \qquad
    S \, = \, \left(\prod_{i=1}^m \. \frac{1}{\alpha_i} \,  \prod_{j=1}^n \. \frac{1}{\beta_j}\right)^{\frac{1}{2}}.
    $$
    In particular, we have
    \[
        \cpc_{\bm\alpha\,\bm\beta}(P_K) \geq \CT_K(\bm\alpha,\bm\beta) \gtrsim N^{-7(m+n)} \cdot \cpc_{\bm\alpha\,\bm\beta}(P_K).
    \]
\end{theorem}

\smallskip

Note that the asymptotics of the bound ratios \ts $N^{-7(m+n)}$ \ts in
Theorem~\ref{t:barvinok-CT-const} and \ts $e^{-4N}$ \ts in Theorem~\ref{t:barvinok-CT-const} \ts
are not directly comparable.  In $\S$\ref{ss:compare-1} we compare the lower bounds directly
and show that our bound in Theorem~\ref{thm:main_general_bound} is sharper. See also
Section~\ref{sec:numerical} for numerical examples.

\smallskip

\begin{remark}[Shapiro's upper bound]\label{rem:main-shapiro}
In~\cite{Sha} (see also~\cite{Bar17}), the Shapiro improves upon
Barvinok's first upper bound by adding a capacity-based \emph{correction term}.
It is best presented in the dual form of the proof of Lemma~\ref{l:Cap1-uniform},
and states:
$$\CT(\bm\alpha,\bm\beta) \, \leq \,
\cpc_{\bm\alpha\,\bm\beta} \,
\min_{\tau \in K_{mn}} \. \prod_{(i,j)\in \tau} \. \frac{1}{1+z_{ij}}\.,
$$
where $Z=(z_{ij})$ is the typical matrix defining capacity in the proof of the lemma,
the minimum is over spanning trees~$\tau$ in the complete bipartite graph $K_{mn}$, and
the product is over all edges in~$\tau$.  Since we concentrate on the lower bounds,
we omit the general $K=(k_{ij})$ case.  We only use this bound in Section~\ref{sec:numerical}
for numerical comparisons.
\end{remark}

\medskip

\subsection{Binary contingency tables}\label{ss:main-binary}
In an important special case of binary contingency tables, Barvinok gives the following bounds.

\smallskip

\begin{theorem}[{Barvinok~\cite[Thm~5]{Bar10}}] \label{thm:barvinok_binary_bound}
Let \ts $\bm\alpha \in \N^m$, \ts $\bm\beta \in \N^n$, such that \ \ts $\sum_i \alpha_i = \sum_j \beta_j$.
Let $K=(k_{ij})$  be an $m \times n$ graphical matrix, i.e.\ $k_{ij}\in \{0,1\}$ for all~$i,j$.
Then:
    \[
        \cpc_{\bm\alpha\,\bm\beta}(P_K) \, \geq \, \BCT_K(\bm\alpha,\bm\beta)
        \, \geq \, \Biggl[\frac{(mn)!}{(mn)^{mn}} \, \prod_{i=1}^m \.
        \frac{(n-\alpha_i)^{n-\alpha_i}}{(n-\alpha_i)!} \.
        \prod_{j=1}^n \, \frac{\beta_j^{\beta_j}}{\beta_j!} \Biggr] \. \cpc_{\bm\alpha\,\bm\beta}(P_K).
    \]
\end{theorem}

\smallskip

For the usual (unrestricted) binary contingency tables, this gives:

\smallskip

\begin{corollary}
Let \ts $\bm\alpha \in \N^m$, \ts $\bm\beta \in \N^n$, such that \ \ts $\sum_i \alpha_i = \sum_j \beta_j$.
Let $K=\bm{1}$ be an $m \times n$  all-ones matrix. Then:
\[
        \cpc_{\bm\alpha\,\bm\beta}(P_K) \, \geq \,
        \BCT(\bm\alpha,\bm\beta) \gtrsim \Biggl[\frac{mn}{(2\pi)^{m+n-1}} \,
        \prod_{i=1}^m \. \frac{1}{n-\alpha_i} \, \prod_{j=1}^n \. \frac{1}{\beta_j}\Biggr]^{\frac{1}{2}}
        \.  \cpc_{\bm\alpha\,\bm\beta}(P_K)\ts,
\]
    where \ts $\min\{\alpha_i,\beta_j\}\to \infty$.
\end{corollary}

\smallskip

Gurvits~\cite{Gur15} was able to improve upon Barvinok's bound in the following
result by proving a better constant for all graphical $m \times n$ matrices $K$.
For the sake of simplicity, he only explicitly provides a bound
in the case of $\bm\alpha = \bm\beta$ are uniform.
He also gives other bounds for non-uniform \ts $\bm\alpha$ and~$\bm\beta$,
in the form that are similar to Theorem~\ref{thm:barvinok_binary_bound}.


\begin{theorem}[{Gurvits~\cite[Thm~5.1]{Gur15}}] \label{thm:main_binary_bound}
Let \ts $\bm\alpha \in \N^m$, \ts $\bm\beta \in \N^n$, such that \ \ts $\sum_i \alpha_i = \sum_j \beta_j$.
Let $K=(k_{ij})$  be an $m \times n$ graphical matrix, i.e.\ $k_{ij}\in \{0,1\}$ for all~$i,j$.
Then:
\[
      \cpc_{\bm\alpha\,\bm\beta}(P_K) \, \geq \, \BCT_K(\bm\alpha,\bm\beta) \,
      \geq \, \Biggl[\prod_{i=2}^m \binom{\lambda_i}{\alpha_i} \.
      \frac{\alpha_i^{\alpha_i} (\lambda_i-\alpha_i)^{\lambda_i-\alpha_i}}{\lambda_i^{\lambda_i}}
      \. \prod_{j=1}^n \binom{\gamma_j}{\beta_j} \. \frac{\beta_j^{\beta_j} (\gamma_j-\beta_j)^{\gamma_j-\beta_j}}{\gamma_j^{\gamma_j}}
      \Biggr] \. \cpc_{\bm\alpha\,\bm\beta}(P_K),
\]
    where \ts $\lambda_i = \sum_j k_{ij}$ \ts and \ts $\gamma_j = \sum_i k_{ij}$\., for all $i$ and~$j$.
\end{theorem}

\smallskip

See~$\S$\ref{ss:finrem-graphs} for a combinatorial interpretation.  In the usual (unrestricted) binary
contingency tables, Gurvits's theorem gives:

\smallskip

\begin{corollary}
Let \ts $\bm\alpha \in \N^m$, \ts $\bm\beta \in \N^n$, such that \ \ts $\sum_i \alpha_i = \sum_j \beta_j$.
Let $K=\bm{1}$ be an $m \times n$  all-ones matrix. Then:
    \[
        \cpc_{\bm\alpha\,\bm\beta}(P_K) \, \geq \, \BCT(\bm\alpha,\bm\beta)
        \, \gtrsim \, \Biggl[\frac{1}{(2\pi)^{m+n-1}} \, \prod_{i=2}^m \. \frac{n}{\alpha_i(n-\alpha_i)} \,
        \prod_{j=1}^n \. \frac{m}{\beta_j(m-\beta_j)}\Biggr]^{\frac{1}{2}}  \cpc_{\bm\alpha\,\bm\beta}(P_K),
    \]
where \ts $\min\{\alpha_i,\beta_j\}\to \infty$.
\end{corollary}

\smallskip


In $\S$\ref{sec:binary_ct_proof}, we give a more straightforward proof
of Theorem~\ref{thm:main_binary_bound} using the same technique
as in the proof of our Main Theorem~\ref{thm:main_general_bound}.
Note that Gurvits's technique in~\cite{Gur15} cannot be applied
to general contingency tables.

\medskip

\subsection{Barvinok's second bound}\label{ss:main-barv-2}
In~\cite{Bar16}, Barvinok gives another upper and lower bound for
$\CT_K(\alpha,\beta)$, similar to the form of Theorem~\ref{thm:barvinok_general_bound},
except the polynomial $P_K$ is replaced by
\[
    H_N(\bm{x},\bm{y}) \, := \, h_N(\bm{x}\cdot\bm{y}) \, =
    \, h_N(\ldots\ts,\ts x_i\ts y_j \ts, \ts \ldots)\.,
\]
where $h_N(\ldots)$ is the complete homogeneous polynomial in $mn$ variables.
Barvinok observes that the coefficients of $H_N(\bm{x},\bm{y})$ are precisely
$\CT(\bm\alpha,\bm\beta)$ for all \ts $\sum_i \alpha_i = \sum_j \beta_j = N$.
Using this, he obtains:

\smallskip

\begin{theorem}[{Barvinok~\cite[Thm~8.4.2]{Bar16}}] \label{thm:barvinok_cs_bound}
    Let \ts $\bm\alpha \in \N^m$ \ts and \ts $\bm\beta \in \N^n$, such that \ \ts
    $\sum_i \alpha_i = \sum_j \beta_j=N$.  Then:
    \[
        \cpc_{\bm\alpha\,\bm\beta}(H_N) \, \geq \, \CT(\bm\alpha,\bm\beta)
        \, \geq \, C_\mathrm{H}(\bm\alpha,\bm\beta) \. \cdot\. \cpc_{\bm\alpha\,\bm\beta}(H_N),
    \]
    where
    \[
        C_\mathrm{H}(\bm\alpha,\bm\beta) \, = \, \binom{N+m-1}{m-1}^{-1} \binom{N+n-1}{n-1}^{-1} \.
        \frac{N!}{N^N} \,\. \max\left\{\prod_{i=1}^m \.\frac{\alpha_i^{\alpha_i}}{\alpha_i!}\.,
        \.\prod_{j=1}^n \.\frac{\beta_j^{\beta_j}}{\beta_j!}\right\}.
    \]
\end{theorem}

\smallskip

Note that Barvinok gives bounds for general $K$ with
entries in $\{0,+\infty\}$, but we suppress this generalization
here for the sake of simplicity.
In~$\S$\ref{ss:compare-2}, we compare our bound with this second
Barvinok's bound.  Namely, we prove that the lower bound
in Theorem~\ref{thm:main_general_bound} is sharper
than that in Theorem~\ref{thm:barvinok_binary_bound}. See also
Section~\ref{sec:numerical} for numerical examples,
and $\S$\ref{ss:finrem-indep}
for the independence heuristic partly motivating this bound.

\bigskip

\section{Random contingency tables} \label{sec:random}

\subsection{The setup}\label{ss:random-setup}
In this section, we give new lower bounds on the probability that a random
contingency table has certain marginals when the entries are drawn from
various random variables. These results parallel similar results given in~\cite{Bar12};
e.g., Theorem 6.3 (2). In what follows, we let $\mu_{\bm\alpha,\bm\beta}$
denote the probability that a random contingency table has marginals
$\bm\alpha,\bm\beta$.

We will also consider the capacity of a different family of polynomials
in what follows. Specifically given choices of $\bm\alpha,\bm\beta$,
a choice of $K=(k_{ij})$ with finite entries, and a choice of $s \in [0,1]$,
we consider \ts $\cpc_{\bm\alpha\,\bm\beta}(Q_{K,s})$, where
\[
Q_{K,s}(\bm{x},\bm{y}) \,  :=  \, \prod_{i=1}^m \.
\prod_{j=1}^n \bigl(s\ts x_i\ts y_j \. + \. (1-s)\bigr)^{k_{ij}}.
\]
We now state our bounds on $\mu_{\bm\alpha,\bm\beta}$ for two specific cases:
when entries are binomial-distributed, and when entries are Poisson-distributed.

\medskip

\subsection{Binomial-distributed entries}\label{ss:random-binomial}
Our first result regarding random contingency tables gives bounds
in the case that the entries of the table are binomial random variables
with parameter \ts $s \in [0,1]$. Specifically given a $K$ with finite entries,
a random contingency table $A$ is sampled by sampling $a_{ij}$ from the set
\ts $\{0,1,\ldots,k_{ij}\}$ \ts with probability given by
\[
    \Pr[a_{ij} = \ell] \, = \, \binom{k_{ij}}{\ell} \. s^\ell \ts (1-s)^{k_{ij}-\ell}.
\]
We now bound \ts $\mu_{\bm\alpha,\bm\beta}$ \ts in this case as follows.

\smallskip

\begin{theorem} \label{thm:general_binom_bound}
    Let \ts $\bm\alpha \in \N^m$, \ts $\bm\beta \in \N^n$, and let $K=(k_{ij})$ be an $m \times n$
    matrix with finite entries \ts $k_{ij}\in \N$. Let $A=(a_{ij})$ be an $m \times n$ random matrix
    where each entry \ts $a_{ij}$ \ts is an independent binomial random variable on \ts
    $\{0,1,\ldots,k_{ij}\}$ \ts with parameter \ts $s \in [0,1]$.
    The probability $\mu_{\bm\alpha,\bm\beta}$ that $A$ has marginals $(\bm\alpha,\bm\beta)$ is bounded by
    \[
        \cpc_{\bm\alpha\,\bm\beta}(Q_{K,s}) \, \geq \, \mu_{\bm\alpha,\bm\beta} \, \geq \,
        \Biggl[\prod_{i=2}^m \binom{\lambda_i}{\alpha_i} \.
        \frac{\alpha_i^{\alpha_i} (\lambda_i-\alpha_i)^{\lambda_i-\alpha_i}}{\lambda_i^{\lambda_i}}
        \, \prod_{j=1}^n \binom{\gamma_j}{\beta_j} \.
        \frac{\beta_j^{\beta_j} (\gamma_j-\beta_j)^{\gamma_j-\beta_j}}{\gamma_j^{\gamma_j}}\Biggr]
        \. \cpc_{\bm\alpha\,\bm\beta}(Q_{K,s}),
    \]
    where $\lambda_i = \sum_j k_{ij}$ and $\gamma_j = \sum_i k_{ij}$.
    \end{theorem}

\smallskip

Note that this is the same constant as is given in the case of counting binary contingency
tables in Theorem~\ref{thm:main_binary_bound}.

\medskip

\subsection{Capacity via typical matrices}\label{ss:random-typical}
In Theorem~\ref{thm:general_binom_bound}, we can replace the expression for
\ts $\cpc_{\bm\alpha\,\bm\beta}(Q_{K,s})$ \ts by a more ostensibly combinatorial
optimization problem. This is very similar to the idea of maximizing an
entropy-like function found in~\cite{BH12}, and in particular
in Lemma~5.3~(2) of~\cite{Bar12}. In those papers,
the optimal input is referred to as the ``typical matrix'' with row sums
$\bm\alpha$ and column sums $\bm\beta$. In the binomial entries case,
we have the following result.

\smallskip

\begin{theorem} \label{t:random-typical}
In notation of Theorem~\ref{thm:general_binom_bound}, we have:
\[
    \cpc_{\bm\alpha\,\bm\beta}(Q_{K,s}) \, = \,
    \sup_{0 \leq M \leq K, \. M\in \mathcal{T}_{\alpha,\beta}}
    \ \prod_{i=1}^m \. \prod_{j=1}^n \, \frac{k_{ij}^{k_{ij}} \.
    s^{m_{ij}} (1-s)^{k_{ij}-m_{ij}}}{m_{ij}^{m_{ij}} (k_{ij}-m_{ij})^{k_{ij}-m_{ij}}}\,,
\]
where the \ts $\sup$ \ts is over all real matrices $M=(m_{ij})$ for which \ts $0 \leq M \leq K$ entrywise,
and $M$ is in the transportation polytope \ts $\mathcal{T}_{\alpha,\beta}$ \ts
of nonnegative real matrices with row sums $\bm\alpha$ and colums sums $\bm\beta$, see~{\rm $\S$\ref{ss:volume-setup}}.
\end{theorem}

\smallskip

Note that, as above, we are able to incorporate the matrix $K$ into this alternate expression.

\medskip

\subsection{Poisson-distributed entries}\label{ss:random-poisson}
Our next result gives bounds in the case that the entries of the table are
Poisson random variables with rate parameter \ts $s > 0$. Specifically,
a random contingency table $A$ is sampled by sampling $a_{ij}$ from
the set $\{0,1,2,3,\ldots\}$ with probability given by
\[
    \Pr[a_{ij} = \ell] \, = \, \frac{s^\ell \ts e^{-s}}{\ell!}\..
\]
For this case, we obtain explicit bounds on the probabilities.

\smallskip

\begin{theorem} \label{thm:general_Poisson_bound}
    Let \ts $\bm\alpha \in \N^m$, \ts $\bm\beta \in \N^n$, such that \ \ts
    $\sum_i \alpha_i = \sum_j \beta_j=N$.  Let $A$ be an $m \times n$
    random matrix where each entry $a_{ij}$ is an independent Poisson random variable
    on \ts $\{0,1,2,3,\ldots\}$ \ts with rate parameter $s > 0$. The probability
    \ts $\mu_{\bm\alpha,\bm\beta}$ that~$A$ has marginals \ts $(\bm\alpha,\bm\beta)$ \ts
    is bounded by
    \[
        \frac{(sN)^N}{e^{smn-N}} \, \prod_{i=1}^m \. \frac{1}{\alpha_i^{\alpha_i}}
        \, \prod_{j=1}^n \. \frac{1}{\beta_j^{\beta_j}}
        \, \geq \, \mu_{\bm\alpha,\bm\beta} \, \geq \,
        \frac{(sN)^N}{e^{smn+N}} \, \prod_{i=1}^m \. \frac{1}{\alpha_i!}
        \, \prod_{j=1}^n \. \frac{1}{\beta_j!}\,.
    \]
\end{theorem}

The proof is based on the fact that the value of \ts $\cpc_{\bm\alpha\,\bm\beta}$
\ts can be explicitly calculated in this case, see~$\S$\ref{sec:Poisson_bound_proof}.

\bigskip

\section{Real stable and denormalized Lorentzian polynomials} \label{sec:Lorenz}

\subsection{Notation}\label{ss:Lorenz-notation}
We use $\C$ and~$\R$ to denote the complex and real numbers, $\R_+=\{x\ge 0\}$,
$\R_{>0}=\{x>0\}$, and $\N=\{0,1,\ldots\}$.  We also use \ts $[n]=\{1,\ldots,n\}$.
We denote by $\partial_i$ the partial derivative \ts $\frac{\partial}{\partial x_i}$.

We use some standard vector shorthand. For vectors $\bm\alpha$ and $\bm\beta$,
we let most standard operations be entrywise; e.g.\ $\bm\alpha \leq \bm\beta$
and $\bm\alpha \pm \bm\beta$. We also define $\bm\alpha! := \prod_i \alpha_i!$,
$\binom{\bm\alpha}{\bm\beta} := \prod_i \binom{\alpha_i}{\beta_i}$, and
$\bm\alpha^{\bm\beta} := \prod_i \alpha_i^{\beta_i}$.
Finally, we let $\bm{1}$ denote the all-ones vector
with length determined by context.

An sequence \ts $\{a_k, 0\le k \le n\}$ \ts of nonnegative real numbers is called \emph{log-concave} if
\ts $a_k^2\ge a_{k-1}a_{k+1}$ \ts for all $1\le k \le n-1$.  Moreover, it is
\emph{ultra-log-concave} if \ts $\bigl\{a_k/\binom{n}{k}, 0\le k \le n\bigr\}$ \ts
is log-concave and has \emph{no internal zeros}, i.e., there are no indices $i<j<k$ for which $a_ia_k\neq 0$ and $a_j=0$.  See~\cite{Bra} and references therein,
for the context behind these properties.

Throughout, we will use Stirling's approximation for factorial:
\[
    \frac{n!}{n^n} \, \approx \, e^{-n} \sqrt{2\pi n}\ts,
\]
which holds asymptotically as $n \to \infty$.
We also have the following bounds which hold for all $n\in \N$\ts:
\[
    \bigl[e^{-n}\ts\sqrt{2\pi n}\bigr] \, \leq \,
    \frac{n!}{n^n} \, \leq \, \frac{e}{\sqrt{2\pi}} \. \bigl[e^{-n}\sqrt{2\pi n}\bigr].
\]

\medskip

\subsection{Real stable and Lorentzian polynomials}\label{ss:Lorenz-realstable}
A polynomial \ts $p \in \C[x_1,\ldots,x_n]$ \ts is said to be \emph{stable}
if it is nonvanishing whenever $\Im(x_j) > 0$ for all~$j$. If further $p$
has real coefficients, then $p$ is said to be \emph{real stable}.
Recall that the \emph{Hessian} of a polynomial \ts $p \in \C[x_1,\ldots,x_n]$ \ts
at \ts $\bm{x} \in \C^n$, is defined as
$$
\He_p(\bm{x}) \. = \. \bigl(\partial_i\partial_j p(x)\bigr)_{i,j=1}^n.
$$
A real symmetric matrix has \emph{Lorentzian signature} if it is nonsingular
and has exactly one positive eigenvalue, i.e., its signature is \ts $(+,-,-,\ldots,-)$.

\smallskip

\begin{definition}[Br\"and\'en--Huh~\cite{BH19}]
    A homogeneous polynomial $p \in \R_+[x_1,\ldots,x_n]$ of degree $d \geq 2$
    is  \emph{strictly Lorentzian} if
    \begin{enumerate}
    \item all coefficients of $p$ are positive, and
    \item for each sequence $1\leq i_1,i_2,\ldots, i_{d-2}\leq n$ and
    $\bm{x}\in \R_{>0}^n$, the Hessian of $\partial_{i_1}\cdots \partial_{i_{d-2}} p$
    has Lorentzian signature at $\bm{x}$.
    \end{enumerate}
If $p$ is the limit (in the Euclidean space of real polynomials of degree
at most $d$ in $n$ variables) of strictly Lorentzian polynomials,
we say that $p$ is \emph{Lorentzian}.
\end{definition}

\smallskip

\begin{proposition}[{\cite[Ex.~5.2]{BH19}}] \label{prop:bivariate_Lorentzian_char}
    A bivariate homogeneous polynomial is Lorentzian if and only if
    its coefficients form an ultra-log-concave sequence.
\end{proposition}

\smallskip

\begin{remark}
    The class of Lorentzian polynomials contains homogenous stable polynomials
    with nonnegative coefficients~\cite{BH19}, which gives an easy
    sufficient condition for a polynomial to be Lorentzian.
\end{remark}

\smallskip

Note that by definition, $\partial_i p$ is Lorentzian whenever $p$ is,
for any~$i$. This gives the following lemma which will be useful for induction.

\smallskip

\begin{lemma} \label{lem:Lorentzian_induction}
    Let \ts $p \in \R_+[x_1,\ldots,x_n]$ \ts be a Lorentzian (resp. real stable) polynomial of degree~$d$, and let us write
    \[
        p(x_1,\ldots,x_n) \, = \, \sum_{k=0}^d \. x_n^{d-k} \ts p_k(x_1,\ldots,x_{n-1})\ts.
    \]
    Then $p_k$ is a Lorentzian (resp. real stable) polynomial of degree~$k$, for all $k \in [d]$.
\end{lemma}

\begin{proof}
    Apply the partial derivative $\partial_n$ exactly $(d-k)$ times to the polynomial~$p$,
    to obtain $q$, which is Lorentzian by definition. Plugging in $x_n=0$ to $q$
    then yields $p_k$ up to constant. To see that plugging in 0 preserves the
    class of Lorentzian polynomials, note that this operation preserves
    the class of stable polynomials with non-negative coefficients,
    see e.g.~\cite[Lemma 2.4]{Wag11}, and it also preserves homogeneity.
    Theorem~6.4 of~\cite{BH19} then implies that this operation also
    preserves the class of Lorentzian polynomials.
\end{proof}

\smallskip

And finally, we have the following by a similar argument.

\smallskip

\begin{lemma}\label{lem:Lorentzian_proj}
    If $p \in \R_+[x_1,x_2, \ldots, x_n]$ is a Lorentzian (resp. real stable) polynomial and $\lambda,\mu>0$,
    then \ts $p(\lambda x_1, \mu x_1, x_3,\ldots, x_n)$ \ts is also Lorentzian (resp. real stable).
\end{lemma}

\begin{proof}
    Since $p \mapsto p(\lambda x_1, \mu x_1, x_3,\ldots, x_n)$ preserves the class
    of real stable polynomials with non-negative coefficients and also preserves homogeneity,
    it also preserves the class of Lorentzian polynomials by Theorem 6.4 of~\cite{BH19}.
\end{proof}

\medskip

\subsection{Denormalized Lorentzian polynomials} \label{sec:denorm_Lorentzian}
Given a polynomial \ts $p(\bm{x}) = \sum_{\bm\mu} p_{\bm\mu} \bm{x}^{\bm\mu}$,
we define its \emph{normalization} as
    \[
        N[p] \, := \, \sum_{\bm\mu} p_{\bm\mu} \frac{\bm{x}^{\bm\mu}}{\bm\mu!}.
    \]
We say a homogeneous polynomial \ts $p \in \R_+[x_1,\ldots,x_n]$ \ts
is \emph{denormalized Lorentzian} if $N[p]$ is Lorentzian.

\smallskip

\begin{proposition}[{\cite[Cor.~6.8]{BH19}}] \label{prop:denorm_product}
Let \ts $p_1,\ldots,p_m \in \R_+[x_1,\ldots,x_n]$ \ts be denormalized Lorentzian polynomials.
Then so is \ts $p_1\cdots p_m$.
\end{proposition}

\smallskip

Recall the function $P_K$ from Section~\ref{sec:main_results}, where $K$
is an $m \times n$ matrix with entries in $\N \cup \{+\infty\}$:
\[
    P_K(\bm{x},\bm{y}) \, := \, \prod_{i=1}^m \prod_{j=1}^n \. \sum_{\ell=0}^{k_{ij}} \. x_i^\ell \ts y_j^\ell.
\]
When $K$ has finite entries, $P_K$ is a polynomial, and we can further define
\[
    \wt{P}_K(\bm{x},\bm{y}) \, := \, y_1^{\sum_i k_{i1}} \ts \cdots \ts y_n^{\sum_i k_{in}} \.
    P_K(\bm{x},\bm{y}^{-1}) \,= \, \prod_{i=1}^m \. \prod_{j=1}^n \.
    \sum_{\ell=0}^{k_{ij}} \. x_i^\ell \ts y_j^{k_{ij}-\ell}.
\]
That is, $\wt{P}_K$ is a product of polynomials of the form \ts
$q(x,y) = x^d + x^{d-1}y + \ts \ldots \ts + y^d$, all of which are
denormalized Lorentzian by Proposition~\ref{prop:bivariate_Lorentzian_char}.
Therefore, $\wt{P}_K$ is denormalized Lorentzian by Proposition~\ref{prop:denorm_product}.
In the next section, we will obtain bounds on the coefficients of
denormalized Lorentzian polynomials, which will translate into
bounds on the number of contingency tables with given marginals.

\smallskip

Before moving on, we give versions of Lemma~\ref{lem:Lorentzian_induction}
and Lemma~\ref{lem:Lorentzian_proj} for denormalized Lorentzian polynomials.

\smallskip

\begin{lemma} \label{lem:denorm_Lorentzian_induction}
    Let \ts $p \in \R_+[x_1,\ldots,x_n]$ \ts be a denormalized Lorentzian polynomial of degree~$d$,
    and let us write
    \[
        p(x_1,\ldots,x_n) \, = \, \sum_{k=0}^d \. x_n^{d-k}\. p_k(x_1,\ldots,x_{n-1}).
    \]
    Then $p_k$ is a denormalized Lorentzian polynomial of degree~$k$, for all $k \in [d]$.
\end{lemma}

\begin{proof}
    The map $p \mapsto p_k$ commutes with $N$ up to scalar for all $k$. The result then follows from
    Lemma~\ref{lem:Lorentzian_induction}.
\end{proof}

\smallskip

\begin{lemma}\label{lem:denorm_Lorentzian_proj}
    Let \ts $p \in \R_+[x_1,x_2, \ldots, x_n]$ \ts be a denormalized Lorentzian polynomial,
    and let \ts $\lambda,\mu>0$.  Then \ts $p(\lambda x_1, \mu x_1, x_3,\ldots, x_n)$ \ts
    is also denormalized Lorentzian.
\end{lemma}

\begin{proof}
Since the class of Lorentzian polynomials is closed under scaling of variables with
positive numbers, so is the class of denormalized Lorentzian polynomials
(since $N$ commutes with scaling). Thus we can assume that \ts $\mu=\lambda=1$.
Let $T$ be the linear operator defined by
\[
    T[p](x_1, \ldots, x_n) \. =  \. p(x_1,x_1,x_3,\ldots,x_n).
\]
We need to prove that the operator \ts $S=N\circ T \circ N^{-1}$ \ts preserves
the class of Lorentzian polynomials. The symbol of $S$ (see~\cite[$\S$6]{BH19}),
is
\[
    S[(x_1+y_1)^d \cdots (x_n+y_n)^d]
        \, = \, (d!)^{-2} \ts (x_3+y_3)^d\cdots (x_n+y_n)^d \.
        \sum_{0\leq k,\ell\leq d} \. \frac {x_1^{k+\ell}}{(k+\ell)!}\.
        \frac {y_1^{d-k}}{(d-k)!}\. \frac {y_2^{d-\ell}}{(d-\ell)!}\..
\]
Up to scalar this is the generating polynomial of an $M$-convex set,
and so it is a Lorentzian polynomial (see~\cite[Thm~7.1]{BH19}).
Therefore, by~\cite[Thm~6.2]{BH19}, the operator~$S$ preserves
the class of Lorentzian polynomials.
\end{proof}

%

\bigskip

\section{Capacity bounds on coefficients} \label{sec:capacity}

\subsection{Preliminaries} \label{ss:capacity-prelim}
Applications of polynomial capacity bounds on stable and Lorentzian
polynomials to combinatorics was pioneered by Gurvits in the mid 2000s.
This began with bounds for the permanent and mixed discriminant
in~\cite{Gur08}, and also includes applications to the mixed volume and to
discrete and computational geometry more generally in~\cite{Gur09b}.

In~\cite{Gur15}, Gurvits used optimal improvements of bounds
from~\cite{Gur09} to prove Theorem~\ref{thm:main_binary_bound}.
The main idea is that one can bound
the coefficients of stable and Lorentzian (i.e.\ strongly log-concave or
completely log-concave) polynomials. Specifically, he applies these bounds
to the polynomial
\[
    \wt{P}_K(\bm{x},\bm{y}) \, := \,
    y_1^{\sum_i k_{i1}} \ts \cdots \ts y_n^{\sum_i k_{in}} \. P_K(\bm{x},\bm{y}^{-1})
    \, = \,
    \prod_{i=1}^m \. \prod_{j=1}^n \. \sum_{\ell=0}^{k_{ij}} \. x_i^\ell \ts y_j^{k_{ij}-\ell}\ts,
\]
where $K$ is a matrix with 0-1 entries. The coefficients of $\wt{P}_K$
are precisely the number of binary contingency tables with given marginals
and entrywise bound matrix~$K$.
(Note that similar bounds can be obtained from the inner product
capacity bounds of~\cite{GL18} and~\cite{AG17}.)

The problem with this approach is that it does not extend to general contingency tables,
since no simple operation applied to $P_K$ yields a stable/Lorentzian polynomial
in that case. To circumvent this, we instead turn to a new approach to deriving
capacity bounds on coefficients of denormalized Lorentzian polynomials.
We can then apply these bounds to $\wt{P}_K$.

\smallskip

Before moving on, we recall the definition of capacity and give a few basic
properties that we will use throughout.

\smallskip

\begin{definition}
    For a polynomial \ts $p \in \R_+[x_1,\ldots,x_n]$ \ts and any non-negative vector \ts
    $\bm\alpha \in \R_+^n$, define
    \[
        \cpc_{\bm\alpha}(p) \, := \, \inf_{\bm{x} > 0} \, \frac{p(\bm{x})}{\bm{x}^{\bm\alpha}}
        \, = \,
        \inf_{x_1,\ldots,x_n > 0} \. \frac{p(x_1,\ldots,x_n)}{x_1^{\alpha_1} \ts  \cdots \ts x_n^{\alpha_n}}\..
    \]
\end{definition}

\smallskip

We also use this definition when $p$ is an analytic function given by a power series
with non-negative coefficients. To handle rational functions which are not analytic,
see below.

\smallskip

\begin{lemma}[{\cite[Lemma~2.16]{GL18}}] \label{lem:cap_linear}
    For any $\bm{c},\bm\alpha \in \R_+^n$ and $m := \sum_{i=1}^n \alpha_k$, we have
    \[
        \cpc_{\bm\alpha}((c_1x_1 + \cdots c_nx_n)^m) \, = \,
        \prod_{i=1}^n \left(\frac{mc_i}{\alpha_i}\right)^{\alpha_i}.
    \]
\end{lemma}

\smallskip

\begin{lemma}[{\cite[Prop.~2.14]{GL18}}] \label{lem:cap_symm}
    Let \ts $p \in \R_+[x_1,\ldots,x_n]$ \ts be a symmetric polynomial,
    and let \ts $\bm\alpha = \gamma \cdot \bm{1} \in \R_+^n$ \ts
    be a multiple of the all-ones vector. Then:
    \[
        \cpc_{\bm\alpha}(p) \. = \. \cpc_{n\gamma}\bigl(p(t,t,\ldots,t)\bigr)\ts.
    \]
\end{lemma}

\smallskip

\begin{lemma}[{\cite[Cor.~5.8]{GL18}}] \label{lem:cap_lim_analytic}
    Let \ts $p_k \in \R_+[x_1,\ldots,x_n]$, for $k \in \N$, be such that \ts
    $p_k \to p$ \ts uniformly on compact sets for some analytic function~$p$.
    Then for any valid \ts $\bm\alpha \in \R_+^n$, we have
    \[
        \cpc_{\bm\alpha}(p) \, = \, \lim_{k \to \infty} \. \cpc_{\bm\alpha}(p_k)\ts.
    \]
\end{lemma}

We make one last comment here about the case when some of the entries of $K$
are $+\infty$. In this case, the function $P_K$ that we care about has some
rational factors of the form \ts $(1-x_iy_j)^{-1}$, and hence is \emph{not} analytic.
So, we need to change notation slightly to account for this:
\[
    \cpc_{(\alpha_i, \beta_j)}\bigl((1-x_iy_j)^{-1}\bigr) \, := \, \inf_{x_i, y_j \in (0,1)} \,
    \frac{(1-x_iy_j)^{-1}}{x_i^{\alpha_i} y_j^{\beta_j}}\..
\]
Note that if we think of \. $(1-x_iy_j)^{-1}$  as the power series \ts
$1 + x_iy_j + (x_iy_j)^2 + \ldots$\., then this makes intuitive sense because
the infimum could only possibly be attained for \ts $x_iy_j \in (0,1)$.
We show that this is a good definition by giving an analogue to
Lemma~\ref{lem:cap_lim_analytic} for this case.

\begin{lemma} \label{lem:cap_lim_rational}
    For every $r \in \R_+$, we have
    \[
        \cpc_r\bigl((1-t)^{-1}\bigr) \, = \,
        \lim_{k \to \infty} \. \cpc_r(1 + t + t^2 + \cdots + t^k)\ts.
    \]
\end{lemma}
\begin{proof}
    The result follows from a standard argument about exchanging \ts $\lim$ \ts
    and \ts $\inf$, since both \ts $t^{-r} (1-t)^{-1}$ \ts and \ts $t^{-r}(1+t+\cdots+t^k)$
    \ts become large near the boundary and convergence is uniform elsewhere.
\end{proof}

\medskip

\subsection{Weighted log-concave coefficients}\label{ss:capacity-weight}
We first prove a capacity bound for bivariate homogeneous polynomials
with weighted log-concave coefficients. Before saying anything more than this,
we define what we mean.

\smallskip

\begin{definition}
    Let \ts $w(x,y) \ts = \ts \sum_{k=0}^n \ts w_k x^k y^{n-k}$ \ts
    and \ts $p(x,y) \ts = \ts \sum_{k=0}^n \ts p_k x^k y^{n-k}$ \ts
    be bivariate homogeneous polynomials with positive coefficients.
    Then we say that polynomial~$p$ is \emph{$w$-log-concave} if \ts
    $\bigl\{p_k/w_k\ts, 0\le k \le n\bigr\}$ \ts is a log-concave sequence.
\end{definition}

\smallskip

In particular, recall that a bivariate homogeneous polynomial $p$ is Lorentzian
if and only if its coefficients form an ultra-log-concave sequence.
That is, $p$ is Lorentzian if and only if it is $(x+y)^n$-log-concave.
Such weightings of log-concave coefficients have been studied in a similar
context by Gurvits in \cite{Gur09} under the name \emph{propagatable sequences}.

\smallskip

We now prove the main lemma of this section, which is a capacity bound
on the coefficients of polynomials with weighted log-concave coefficients.

\smallskip

\begin{lemma} \label{lem:weighted_log_concave_capacity}
    Let \ts $w(x,y) \ts = \ts \sum_{k=0}^n w_k x^k y^{n-k}$ \ts and \ts
    $p(x,y) = \sum_{k=0}^n p_k x^k y^{n-k}$ \ts be bivariate homogeneous polynomials,
    such that $p$ is $w$-log-concave. Then for all $k \in [n]$, we have
    \[
        \frac{p_k}{\cpc_{(k,n-k)}(p)} \, \geq \, \frac{w_k}{\cpc_{(k,n-k)}(w)}\..
    \]
    Furthermore, this bound is sharp for every fixed $k$ and~$w$.
\end{lemma}

\begin{proof}
    We need to compute
    \[
        C_{k,n-k} \, := \, \sup_{\text{l.c.}\ts\bm{a}} \ \inf_{x,y > 0} \,
        \frac{1}{a_k x^k y^{n-k}}\, \sum_{j=0}^n \. w_j a_j x^j y^{n-j}\.,
    \]
    where the $\sup$ is over all positive log-concave sequences \ts
    $\bm{a} = (a_0, \ldots, a_n)$. Since $p$ is $w$-log-concave,
    the result is then equivalent to
    \[
        C_{k,n-k} \. = \. \cpc_{(k,n-k)}(w).
    \]
    To prove this, first note that
    \[
        C_{k,n-k} \. = \, \sup_{\text{l.c.}\ts\bm{a}} \ \inf_{x,y > 0} \,
        \frac{1}{a_k x^k y^{n-k}} \, \sum_{j=0}^n \. w_j a_j x^j y^{n-j}
        \, = \,
        \sup_{\substack{\text{ l.c.}\ts\bm{a} > 0 \\ a_k = 1}} \ \inf_{x > 0} \, \Biggl(\sum_{j=0}^n \. w_j a_j x^{j-k}\Biggr).
    \]
    For log-concave $\bm{a} = (a_0, \ldots, a_n)$ with $a_k = 1$, we have the inequalities
    \[
        a_{k-j} \. \leq \. a_{k-1}^j \quad \text{and} \quad a_{k+j} \leq a_{k+1}^j
    \]
    for every valid $j \geq 0$. This implies
    \[
        C_{k,n-k}\, = \, \sup_{a_{k-1} \ts a_{k+1} \leq 1} \, \inf_{x > 0} \,
        \Biggl[\Biggl(\sum_{j=0}^{k-1} \. w_j a_{k-1}^{k-j} x^{j-k}\Biggr) \. + \. w_k \. + \.
        \Biggl(\sum_{j=k+1}^n w_j a_{k+1}^{j-k} x^{j-k}\Biggr)\Biggr].
    \]
    We can further restrict to \ts $a_{k-1} a_{k+1} = 1 \iff a_{k-1} = a_{k+1}^{-1}$ \ts in the $\sup$, which implies
    \[
        C_{k,n-k} \, = \,\sup_{a_{k+1} > 0} \. \inf_{x > 0} \. \Biggl(\sum_{j=0}^n \.
        w_j a_{k+1}^{j-k} x^{j-k}\Biggr) \, =  \, \inf_{x > 0} \. \Biggl(\sum_{j=0}^n \. w_j x^{j-k}\Biggr).
    \]
    Therefore \ts $C_{k,n-k} = \cpc_{(k,n-k)}(w)$, which implies the result.
    Sharpness of the bound is then achieved by setting $p = w$.
\end{proof}

\smallskip

Using this lemma, we derive corollaries for specific polynomials pertinent to the polynomial $\wt{P}_K$.
Note that in this first result we make a simplification to obtain a nice expression for the bound,
and therefore the bound is not sharp.

\smallskip

\begin{corollary} \label{cor:denorm_Lor_bivariate_capacity}
    Let \ts $p(x,y) = \sum_{k=0}^n p_k x^k y^{n-k}$ \ts be such that \ts
    $p_0, \ldots, p_n$ \ts is a positive log-concave sequence. We have:
    \[
        \frac{p_k}{\cpc_{(k,n-k)}(p)} \, \geq \, \max\.\left\{\frac{k^k}{(k+1)^{k+1}}, \. \frac{(n-k)^{n-k}}{(n-k+1)^{n-k+1}}\right\},
    \]
    for every \ts $k \in [n]$.
    Further, this bound is sharp up to a factor of \ts $\frac{e}{2}$ \ts for every fixed \ts $k \in [n]$.
\end{corollary}

\begin{proof}
    We compute
    \[
        \cpc_{(k,n-k)}\. \Biggl(\sum_{j=0}^n x^j y^{n-j}\Biggr) \, = \,
        \inf_{x > 0} \, \sum_{j=0}^n \ts x^{j-k} \, \leq \, \inf_{x \in (0,1)} \. \bigl[x^{-k}(1-x)^{-1}\bigr]\ts.
    \]
    Basic calculus gives
    \[
        \sup_{x \in (0,1)} \. \bigl[x^k - x^{k+1}\bigr] \, = \, \frac{k^k}{(k+1)^{k+1}}\,.
    \]
   Combined with the previous lemma this implies
    \[
        p_k \. \geq \, \frac{\cpc_{(k,n-k)}(p)}{\cpc_{(k,n-k)}\left(\sum_{j=0}^n x^j y^{n-j}\right)}
        \, \geq \, \frac{k^k}{(k+1)^{k+1}} \. \cpc_{(k,n-k)}(p)\ts.
    \]
    We also have
    \[
        \cpc_{(k,n-k)}\Biggl(\sum_{j=0}^n \ts x^j y^{n-j}\Biggr) \, = \,
        \inf_{y > 0} \, \sum_{j=0}^n \ts y^{(n-j)-(n-k)} \, \leq \,
        \inf_{y \in (0,1)} \left[y^{n-k}(1-y)\right]^{-1}.
    \]
    The same argument then implied the bound in the corollary and finishes the proof of the first part.

\smallskip

    For the second part, first note that for \ts $m := \min(k,n-k)$ \ts we have
    \[
        \cpc_{(k,n-k)}\Biggl(\sum_{j=0}^n x^j y^{n-j}\Biggr) \, \geq \,
        \cpc_{(m,m)}\Biggl(\sum_{j=0}^{2m} x^j y^{2m-j}\Biggr) \. = \, 2m+1,
    \]
    by symmetry and Lemma~\ref{lem:cap_symm}. Therefore,
    \[
        \frac{m^m}{(m+1)^{m+1}} \, \leq \,
        \Biggl[\cpc_{(k,n-k)}\Biggl(\sum_{j=0}^n x^j y^{n-j}\Biggr)\Biggr]^{-1}
        \, \leq \, \frac{1}{2m+1}\..
    \]
    By a calculus argument, we further have
    \[
        \frac{2}{e} \. \cdot\. \frac{1}{2m+1}\, \leq \,\frac{m^m}{(m+1)^{m+1}}\.,
    \]
    and this completes the proof.
\end{proof}

\smallskip

\begin{corollary} \label{cor:Lorentzian_bivariate_capacity}
    Let $p(x,y) = \sum_{k=0}^n p_k x^k y^{n-k}$ be a bivariate Lorentzian polynomial. For each $k \in [n]$, we have
    \[
        \frac{p_k}{\cpc_{(k,n-k)}(p)} \, \geq \, \binom{n}{k} \. \frac{k^k (n-k)^{n-k}}{n^n}\,.
    \]
    Further, this bound is sharp for every fixed \ts $k \in [n]$.
\end{corollary}

\begin{proof}
    Recall that a bivariate homogeneous polynomial is Lorentzian if and only if it is $(x+y)^n$-log-concave.
    By Lemma~\ref{lem:cap_linear}, we have
    \[
        \cpc_{(k,n-k)}\Biggl(\sum_{k=0}^n \. \binom{n}{k} x^k y^{n-k}\Biggr) \, = \,
        \cpc_{(k,n-k)}\bigl((x+y)^n\bigr) \, = \, \frac{n^n}{k^k (n-k)^{n-k}}\..
    \]
    The $k$-\text{th} coefficient of $(x+y)^n$ is $\binom{n}{k}$, and this completes the proof.
\end{proof}

\smallskip

\subsection{Real stable and denormalized Lorentzian polynomials}\label{ss:capalicy-lorenz}
In this section, we emulate Gurvits's proof of Theorem~5.1 of~\cite{Gur15}
to obtain bounds for real stable and denormalized Lorentzian polynomials.
First, we prove our main bound on coefficients of denormalized Lorentzian polynomials.

\smallskip

\begin{theorem} \label{thm:denorm_Lorentzian_bound}
    Let \ts $p \in \R_+[x_1,\ldots,x_n]$ \ts be a denormalized Lorentzian polynomial
    of degree $d$, given by
    $$
        p(x_1,\ldots,x_n) \, = \, \sum_{\mu_1\ts + \ts\ldots\ts + \ts\mu_n = d} \. p_{\bm\mu} \ts \bm{x}^{\bm\mu}\ts.
    $$
    Let $d_i$ be the degree of $x_i$ in
    $$
        \left.\partial_{i+1}^{\alpha_{i+1}} \cdots \ts \partial_{n}^{\alpha_{n}} \ts p\right|_{x_{i+1}=\ts\ldots\ts=x_n=0} \ ,
        \quad \text{for all} \quad 1\le i \le n-1,
    $$
    and let $d_n$ be the degree of $x_n$ in~$p$. Then, for all \ts $\bm\alpha \in \N^n$,
    such that \ \ts $\alpha_1 + \cdots + \alpha_n = d$, we have:
    \begin{align*}
        \frac{p_{\bm\alpha}}{\cpc_{\bm\alpha}(p)} \, 
        \geq \,
        \left[\prod_{i=2}^n \cpc_{(d_i-\alpha_i, \alpha_i)}\left(\sum_{k=0}^{d_i} x^k y^{d_i-k}\right)\right]^{-1} 
            \geq \, \prod_{i=2}^n \. \max\left\{\frac{\alpha_i^{\alpha_i}}{(\alpha_i+1)^{\alpha_i+1}}, \.
            \frac{(d_i-\alpha_i)^{d_i-\alpha_i}}{(d_i-\alpha_i+1)^{d_i-\alpha_i+1}}\right\}.
    \end{align*}
\end{theorem}

\smallskip

\begin{proof}
    The proof is by induction over $n \geq 2$. The case $n=2$ is Lemma~\ref{lem:weighted_log_concave_capacity}
    and Corollary~\ref{cor:denorm_Lor_bivariate_capacity}. Let $n>2$, and write
    $$
        p(x_1,\ldots,x_n) \, = \, \sum_{i=0}^{d} \. x_n^{d-i} \ts p_i(x_1,\ldots,x_{n-1}).
    $$
    Then for positive $y_1,\ldots, y_{m-1}$, Lemma~\ref{lem:denorm_Lorentzian_proj} implies that
    $$
        p(y_1 t, \ldots, y_{n-1} t, s) \, = \, \sum_{i=0}^{d} \. s^{d-i} t^i p_i(y_1,\ldots,y_{n-1})
    $$
    is denormalized Lorentzian. Now for $s,t\ge 0$, we have:
      \begin{align*}
        \cpc_{\bm\alpha}(p) \, \leq \, \frac {p(y_1 t, \ldots, y_{n-1} t, s)}{(ty_1)^{\alpha_1} \cdots (ty_{n-1})^{\alpha_{n-1}}s^{\alpha_n}}
        \, = \, \frac {\sum_{i=0}^{d_n} \. s^{d_n-i} t^i p_i(y_1,\ldots,y_{n-1})}{ t^{d_n-\alpha_n}s^{\alpha_n}y_1^{\alpha_1}\cdots y_{n-1}^{\alpha_{n-1}}}\..
      \end{align*}
 Clearly the polynomial \ts $\sum_{i=0}^{d_n} \. s^{d_n-i} t^i p_i(y_1,\ldots,y_{n-1})$ \ts is also denormalized Lorentzian.
 Thus by Lemma~\ref{lem:weighted_log_concave_capacity}, we obtain:
    \begin{align*}
        \cpc_{\bm\alpha}(p) \, & \leq \, \. \frac{\cpc_{(d_n-\alpha_n,\alpha_n)}
        \left( \sum_{i=0}^{d_n} \ts s^{d_n-i}t^{i} \ts p_i(y_1,\ldots,y_{n-1})\right)}{y_1^{\alpha_1}\cdots y_{n-1}^{\alpha_{n-1}}}
        \\
        & \leq \, \. \frac{\cpc_{(d_n-\alpha_n, \alpha_n)}\left(\sum_{k=0}^{d_n} x^k y^{d_n-k}\right) \cdot p_{\alpha_n}(y_1,\ldots, y_{n-1})}{y_1^{\alpha_1}\cdots y_{n-1}^{\alpha_{n-1}}}\,.
    \end{align*}
    Therefore, by Corollary~\ref{cor:denorm_Lor_bivariate_capacity}, we conclude:
    \begin{align*}
        \cpc_{(\alpha_1,\ldots, \alpha_{n-1})}(p_{\alpha_n})\,  &\geq \,\.
        \frac{\cpc_{\bm\alpha}(p)}{\cpc_{(d_n-\alpha_n, \alpha_n)}\left(\sum_{k=0}^{d_n} x^k y^{d_n-k}\right)} \\
            &\geq \,\. \max\left\{\frac{\alpha_i^{\alpha_i}}{(\alpha_i+1)^{\alpha_i+1}}, \.
            \frac{(d_i-\alpha_i)^{d_i-\alpha_i}}{(d_i-\alpha_i+1)^{d_i-\alpha_i+1}}\right\} \.\cpc_{\bm\alpha}(p)\ts.
    \end{align*}
    The result follows by induction.
\end{proof}

Essentially the same proof also works for real stable polynomials, giving the same bound achieved by Gurvits. The proof we give here is similar to Gurvits's, but we have put it in our language for the sake of clarity and easy comparison to the proof of Theorem~\ref{thm:denorm_Lorentzian_bound}.


\begin{theorem}[{\cite[Thm~5.1]{Gur15}}] \label{thm:Lorentzian_bound}
    Let \ts $p \in \R_+[x_1,\ldots,x_n]$ \ts be a real stable polynomial of degree $d$, given by
    $$
        p(x_1,\ldots,x_n)\, = \, \sum_{\mu_1 + \cdots + \mu_n = d} \. p_{\bm\mu} \ts \bm{x}^{\bm\mu}.
    $$
    Let $d_i$ be the degree of $x_i$ in
    $$
        \left.\partial_{i+1}^{\alpha_{i+1}} \cdots \ts \partial_{n}^{\alpha_{n}} \ts p\right|_{x_{i+1}=\ldots=x_n=0}\,, \ \ \ i=1,\ldots, n-1,
    $$
    and $d_n$, the degree of $x_n$ in $p$. For any $\bm\alpha \in \N_+^n$ such that $\alpha_1 + \ldots + \alpha_n = d$, we have
    $$
        \frac{p_{\bm\alpha}}{\cpc_{\bm\alpha}(p)} \, \geq \,
        \prod_{i=2}^n \. \binom{d_i}{\alpha_i} \. \frac{\alpha_i^{\alpha_i}(d_i-\alpha_i)^{d_i-\alpha_i}}{d_i^{d_i}}\,.
    $$
\end{theorem}

\begin{proof}
    The proof is by induction over $n \geq 2$. The case of $n=2$ is
    Corollary~\ref{cor:Lorentzian_bivariate_capacity}.
    Now, every step of the induction of the proof of Theorem~\ref{thm:denorm_Lorentzian_bound}
    then holds for real stable polynomials, with \ts $\sum_{i=0}^{d_n} x^i y^{d_n-i}$ \ts
    replaced by $(x+y)^{d_i}$. The main difference is that in the second to last step
    we apply Corollary~\ref{cor:Lorentzian_bivariate_capacity} to get
    \[
        \cpc_{\bm\alpha}(p) \, \leq \,
        \binom{d_n}{\alpha_n}^{-1} \. \frac{d_n^{d_n}}{\alpha_n^{\alpha_n} (d_n-\alpha_n)^{d_n-\alpha_n}}
        \. \cdot \. \frac{p_{\alpha_n}(y_1,\ldots,y_{n-1})}{y_1^{\alpha_1} \cdots \ts y_n^{\alpha_n}}\,,
    \]
    for all $y_1,\ldots,y_{n-1} > 0$. This implies
    \[
        \cpc_{(\alpha_1,\ldots, \alpha_{n-1})}(p_{\alpha_n}) \,\geq \,
        \binom{d_n}{\alpha_n} \. \frac{\alpha_n^{\alpha_n} (d_n-\alpha_n)^{d_n-\alpha_n}}{d_n^{d_n}}
        \. \cdot \. \cpc_{\bm\alpha}(p)\.,
    \]
which proves the step of induction.
\end{proof}

\bigskip

\section{Proofs of the results}\label{sec:proofs}

In this section we prove the results in Sections~\ref{sec:main_results} and~\ref{sec:random}.
We obtain bounds on $\CT_K(\bm\alpha,\bm\beta)$ for various $K$, and on probabilities
that a random contingency table will have marginals $(\bm\alpha,\bm\beta)$
when the entries are chosen from binomial and Poisson distributions.
The proofs of these facts all have the same form:

(1) determine a polynomial whose coefficients hold the information that we want to bound,

(2) transform that polynomial 
until it is real stable or denormalized Lorentzian, and then

(3) apply the capacity bounds of the previous section.

\medskip

\subsection{Proof of Theorem~\ref{thm:main_general_bound} and Corollary~\ref{c:main-CT}}
\label{sec:gen_ct_proof}
%
Recall that \ts $\CT_K(\bm\alpha,\bm\beta)$ \ts is the coefficient of \ts
$\bm{x}^{\bm\alpha} \bm{y}^{\bm\beta}$ \ts in the polynomial
\[
    P_K(\bm{x},\bm{y}) \, := \, \prod_{i=1}^m \. \prod_{j=1}^n \. \sum_{\ell=0}^{k_{ij}} \ts x_i^\ell y_j^\ell
\]
for every $\bm\alpha$, $\bm\beta$, and every~$K=(k_{ij})$. When the entries of $K$ are finite,
we invert the $\bm{y}$ variables to get
\[
    \wt{P}_K(\bm{x},\bm{y}) \, := \, y_1^{\sum_i k_{i1}} \cdots \ts y_n^{\sum_i k_{in}} \. P_K(\bm{x},\bm{y}^{-1})
    \, = \, \prod_{i=1}^m \. \prod_{j=1}^n \. \sum_{\ell=0}^{k_{ij}} \. x_i^\ell y_j^{k_{ij}-\ell}\ts.
\]
The polynomial $\wt{P}_K$ is a product of denormalized Lorentzian polynomials.
By Proposition~\ref{prop:denorm_product}, this implies $\wt{P}_K$ is also denormalized Lorentzian.
Therefore, we can apply Theorem~\ref{thm:denorm_Lorentzian_bound} to $\wt{P}_K$.
Since $\wt{P}_K$ is of degree \ts $\lambda_i := \sum_j k_{ij}$ \ts in $x_i$ for all~$i$,
and of degree \ts $\gamma_j := \sum_i k_{ij}$ \ts in $y_j$ for all~$j$,
we obtain the following for any valid \ts $\bm\alpha,\bm\beta$:
\[
    \frac{[\wt{P}_K]_{\bm\alpha\,(\bm\gamma-\bm\beta)}}{\cpc_{\bm\alpha\,(\bm\gamma-\bm\beta)}(\wt{P}_K)}
    \, \geq \,  \prod_{i=2}^m \frac{\alpha_i^{\alpha_i}}{(\alpha_i+1)^{\alpha_i+1}} \,
    \prod_{j=1}^n \. \frac{\beta_j^{\beta_j}}{(\beta_j+1)^{\beta_j+1}}\..
\]
Here \ts $[\wt{P}_K]_{\bm\alpha\,\bm\beta}$ \ts denotes the coefficient of $\wt{P}_K$
corresponding to the monomial \ts $\bm{x}^{\bm\alpha} \bm{y}^{\bm\beta}$.
Finally, it is straightforward to see that
\[
    \frac{\CT_K(\bm\alpha,\bm\beta)}{\cpc_{\bm\alpha\,\bm\beta}(P_K)} \, = \,
    \frac{[P_K]_{\bm\alpha\,\bm\beta}}{\cpc_{\bm\alpha\,\bm\beta}(P_K)} \, = \,
    \frac{[\wt{P}_K]_{\bm\alpha\,(\bm\gamma-\bm\beta)}}{\cpc_{\bm\alpha\,(\bm\gamma-\bm\beta)}(\wt{P}_K)}\,.
\]
We now simplify this bound. Note first that for any $k \in \N$, we have:
\[
    \frac{k^k}{(k+1)^{k+1}} \, = \,
    \frac{1}{k+1} \left(\frac{k}{k+1}\right)^k \, \geq \, \frac{1}{e(k+1)}\..
\]
Combining this with the above bound gives
\begin{align*}
    \frac{\CT_K(\bm\alpha,\bm\beta)}{\cpc_{\bm\alpha\,\bm\beta}(P_K)} \, & \geq \
    \prod_{i=2}^m \. \frac{\alpha_i^{\alpha_i}}{(\alpha_i+1)^{\alpha_i+1}} \.
    \prod_{j=1}^n \. \frac{\beta_j^{\beta_j}}{(\beta_j+1)^{\beta_j+1}} \\
        &\geq \ \frac{1}{e^{m+n-1}} \. \prod_{i=2}^m \frac{1}{\alpha_i+1} \. \prod_{j=1}^n \. \frac{1}{\beta_j+1}\..
\end{align*}
Finally if $K$ has some entries which are $+\infty$, then we can choose large finite numbers for the those entries,
apply the previous argument, and limit to $+\infty$ (see Lemma~\ref{lem:cap_lim_analytic}). \ $\sq$

\begin{remark} \label{rem:stronger_general_bound}
    Note that we did not use the full strength of Theorem~\ref{thm:denorm_Lorentzian_bound} here,
    which allows us to make the following replacements:
    \[
        \frac{\alpha_i^{\alpha_i}}{(\alpha_i+1)^{\alpha_i+1}} \ \longrightarrow \
        \max\left\{\frac{\alpha_i^{\alpha_i}}{(\alpha_i+1)^{\alpha_i+1}}, \,
        \frac{(\lambda_i-\alpha_i)^{\lambda_i-\alpha_i}}{(\lambda_i-\alpha_i+1)^{\lambda_i-\alpha_i+1}}\right\}
    \]
    and
    \[
        \frac{\beta_j^{\beta_j}}{(\beta_j+1)^{\beta_j+1}} \ \longrightarrow \
        \max\left\{\frac{\beta_j^{\beta_j}}{(\beta_j+1)^{\beta_j+1}}, \,
        \frac{(\gamma_j-\beta_j)^{\gamma_j-\beta_j}}{(\gamma_j-\beta_j+1)^{\gamma_j-\beta_j+1}}\right\}.
    \]
    Via the above simplification, we then obtain the stronger bound
    \[
        \frac{\CT_K(\bm\alpha,\bm\beta)}{\cpc_{\bm\alpha\,\bm\beta}(P_K)} \, \geq \,
        \frac{1}{e^{m+n-1}} \. \prod_{i=2}^m \. \frac{1}{\min\{\alpha_i, \lambda_i-\alpha_i\}+1} \.
        \prod_{j=1}^n \. \frac{1}{\min\{\beta_j, \gamma_j-\beta_j\}+1}\.,
    \]
    where \ts $\lambda_i := \sum_j k_{ij}$ \ts for all~$i$, and \ts
    $\gamma_j := \sum_i k_{ij}$ \ts for all~$j$.
\end{remark}

\medskip

\subsection{Proof of Theorem~\ref{thm:almost_const}}
In this section, we obtain a simply exponential approximation factor for
$\CT(\bm\alpha,\bm\beta)$ in the case that $\alpha_i, \beta_j \leq c$ for $i,j \geq 2$.
Letting $N := \sum_{i=1}^m \alpha_i = \sum_{j=1}^n \beta_i$ (which is dominated by the
value of $\alpha_1$ and $\beta_1$), we will compute $\CT_K(\bm\alpha,\bm\beta)$,
where we may assume $K$ to be the matrix with each entry defined by
$k_{ij} := \min(\alpha_i,\beta_j)$. By Remark~\ref{rem:stronger_general_bound},
we obtain the bound
\[
\begin{split}
    \frac{\CT_K(\bm\alpha,\bm\beta)}{\cpc_{\bm\alpha\,\bm\beta}(P_K)} \, &\geq \,
    \frac{1}{e^{m+n-1}} \, \prod_{i=2}^m \. \frac{1}{\min\{\alpha_i, \lambda_i-\alpha_i\}+1} \,
    \prod_{j=1}^n \. \frac{1}{\min\{\beta_j, \gamma_j-\beta_j\}+1} \\
        &\geq \, \frac{1}{e^{m+n-1}} \. \cdot \. \frac{1}{\gamma_1-\beta_1+1} \,
        \prod_{i=2}^m \. \frac{1}{c+1} \, \prod_{j=2}^n \frac{1}{c+1} \\
        &\geq \frac{1}{e^{m+n-1}} \. \cdot \. \frac{1}{(m-1)c+1} \. \cdot \. \frac{1}{(c+1)^{m+n-2}} \\
        &\geq \frac{1}{(m+n-1) \. \cdot \. (e(c+1))^{m+n-1}}\,,
\end{split}
\]
where \ts $\lambda_i := \sum_j k_{ij}$ for all~$i$, and $\gamma_j := \sum_i k_{ij}$ for all~$j$.
Since \ts $\CT_K(\bm\alpha,\bm\beta) = \CT(\bm\alpha,\bm\beta)$ \ts in this case, we obtain
\[
    \cpc_{\bm\alpha\,\bm\beta}(P_K) \, \geq \, \CT(\bm\alpha,\bm\beta) \,
    \geq \, \frac{1}{(m+n-1) \. \cdot \. (e(c+1))^{m+n-1}} \. \cdot \. \cpc_{\bm\alpha\,\bm\beta}(P_K)\ts.
\]
This finishes the proof.  \ $\sq$

\medskip

\subsection{Proof of Theorem~\ref{thm:general_binom_bound}} \label{sec:binom_bound_proof}
We now prove the probability bound in the case where the $(i,j)$-\text{th} entry
of the table is sampled according to a binomial distribution on \ts $\{0,1,\ldots,k_{ij}\}$
with parameter $s \in [0,1]$. The probability of selecting a contingency table
with marginals $(\bm\alpha,\bm\beta)$ in this case is given by the coefficient of \ts
$\bm{x}^{\bm\alpha} \bm{y}^{\bm\beta}$ \ts in the polynomial
\[
    Q_{K,s}(\bm{x},\bm{y}) \, := \, \prod_{i=1}^m \. \prod_{j=1}^n \.
    \bigl(sx_iy_j + (1-s)\bigr)^{k_{i,j}}
\]
for all \ts $\bm\alpha$, $\bm\beta$, and $K=(k_{ij})$ with finite entries.
Note that \ts $Q_{K,s}(\bm{1},\bm{1}) = 1$.
We can invert the $\bm{y}$ variables to get
\[
    \wt{Q}_{K,s}(\bm{x},\bm{y}) \, := \, y_1^{\sum_i k_{i1}} \cdots
    y_n^{\sum_i k_{in}} \ts Q_{K,s}(\bm{x},\bm{y}^{-1}) \,= \,
    \prod_{i=1}^m \. \prod_{j=1}^n \. \bigl(sx_i + (1-s)y_j\bigr)^{k_{ij}}.
\]
This polynomial is real stable, and we can apply Theorem~\ref{thm:Lorentzian_bound}.
Since $\wt{Q}_{K,s}$ is of degree \ts
$\lambda_i := \sum_j k_{ij}$ \ts in $x_i$ for all~$i$, and of degree \ts
$\gamma_j := \sum_i k_{ij}$ \ts in $y_j$ for all~$j$, we obtain the following
for all valid \ts $\bm\alpha,\bm\beta$\ts:
\[
    \frac{[\wt{Q}_{K,s}]_{\bm\alpha,(\bm\gamma-\bm\beta)}}{\cpc_{\bm\alpha,\bm(\gamma-\bm\beta)}(\wt{Q}_{K,s})}
    \, \geq \,
    \prod_{i=2}^m \. \binom{\lambda_i}{\alpha_i} \. \frac{\alpha_i^{\alpha_i}(\lambda_i-\alpha_i)^{\lambda_i-\alpha_i}}{\lambda_i^{\lambda_i}}
    \, \prod_{j=1}^n \.\binom{\gamma_j}{\beta_j} \. \frac{\beta_j^{\beta_j}(\gamma_j-\beta_j)^{\gamma_j-\beta_j}}{\gamma_j^{\gamma_j}}\..
\]
Here, $[\wt{Q}_{K,s}]_{\bm\alpha,\bm\beta}$ denotes the coefficient of \ts $\wt{Q}_{K,s}$ \ts
corresponding to the monomial $\bm{x}^{\bm\alpha} \bm{y}^{\bm\beta}$.
It is then straightforward to see that
\[
    \frac{[Q_{K,s}]_{\bm\alpha,\bm\beta}}{\cpc_{\bm\alpha,\bm\beta}(Q_{K,s})}
    \, = \, \frac{[\wt{Q}_{K,s}]_{\bm\alpha,(\bm\gamma-\bm\beta)}}{\cpc_{\bm\alpha,(\bm\gamma-\bm\beta)}(\wt{Q}_{K,s})}\..
\]
This gives the desired bound. \ $\sq$

\medskip

\subsection{Proof of Theorem~\ref{thm:main_binary_bound}} \label{sec:binary_ct_proof}
The bound in this case follows from the binomial-distributed case
($\S$\ref{sec:binom_bound_proof} above), up to scalar when $K$ is a 0-1 matrix
and \ts $s = \frac{1}{2}$.  The details are straightforward.   \ $\sq$

\medskip

\subsection{Proof of Theorem~\ref{t:random-typical}} \label{ss:proof-typical}
The equality follows from the same sort of arguments used
to prove Lemma~5.3~(2) of~\cite{Bar12}. However, another proof can be
given using the following nice capacity-theoretic result, which we
present for completeness.

\smallskip

\begin{proposition}
    Given \ts $p_1,\ldots,p_m \in \R_+[x_1,\ldots,x_n]$ \ts and \ts $\bm\alpha \in \R_+^n$, we have:
    \[
        \cpc_{\bm\alpha}\Biggl(\prod_{k = 1}^m \. p_k \Biggr) \, = \,
        \sup_{\substack{\bm\beta^1\ts,\ts\ldots\ts,\ts\bm\beta^m \in \R_+^n \\
        \bm\beta^1\ts + \. \ldots \. + \ts \bm\beta^m = \ts \bm\alpha}} \
        \prod_{k = 1}^m \. \cpc_{\bm\beta^k}(p_k).
    \]
\end{proposition}

\begin{proof}[Proof outline]
    Define \ts $p := \prod_{k=1}^m p_k$.
    First, the case where \ts
    $\left.\nabla \log(p)\right|_{\bm{x}=\bm{1}} = \bm\alpha$ \ts
    follows from the fact that \ts $\cpc_{\bm\gamma}(p)$ \ts is maximized over \ts
    $\bm\gamma \in \R_+^n$ \ts at \ts $\bm\gamma = \left.\nabla \log(p)\right|_{\bm{x}=\bm{1}}$, and \ts $\cpc_{\bm\alpha}(p) = p(\bm{1})$ \ts in this case (see Fact~2.10 of~\cite{GL18}).
    Then, to handle $\bm\alpha$ in the relative interior of the Newton polytope of~$p$,
    one can choose $\bm{r} > 0$ such that \ts
    $\left.\nabla \log\bigl(p(\bm{r} \cdot \bm{x})\bigr)\right|_{\bm{x}=\bm{1}} = \bm\alpha$.
    The result then follows from the first case. Since the case of $\bm\alpha$
    outside the Newton polytope of $p$ is trivial (because the capacity is~0),
    we only need to handle the case when $\bm\alpha$ is on the relative boundary
    of the Newton polytope of~$p$. This can be done by a limiting argument; the
    details are straightforward.
\end{proof}

Once we have this result, Theorem~\ref{t:random-typical} follows from a straightforward application to the polynomial \ts $Q_{K,s}$, using Lemma~\ref{lem:cap_linear} to obtain explicit expressions for the capacity of the terms of the product.

\medskip

\subsection{Proof of Theorem~\ref{thm:general_Poisson_bound}} \label{sec:Poisson_bound_proof}
We now prove the probability bound in the case where the entries of the table are sampled
according to the Poisson distribution on $\{0,1,2,\ldots\}$ with parameter $s > 0$.
The probability of selecting a contingency table with marginals $(\bm\alpha,\bm\beta)$
in this case is given by the coefficient of $\bm{x}^{\bm\alpha} \bm{y}^{\bm\beta}$ in the power series of
\[
    Q_{\infty,s}(\bm{x},\bm{y})\, := \, \prod_{i=1}^m \. \prod_{j=1}^n \. e^{sx_iy_j - s}
    \quad \text{for all \ $\bm\alpha$ \ and \ $\bm\beta$.}
\]
Note that \ts $Q_{\infty,s}(\bm{1},\bm{1}) = 1$.
Because this is not a polynomial, we can't invert the $\bm{y}$ variables as we have done above.
Instead, we view this case as a limit of the binomial case.  In particular, note that
\[
    e^{sx_iy_j\ts -\ts s} \, = \, \lim_{d \to \infty} \, \frac{1}{e^s}  \left(\frac{sx_iy_j}{d} + 1\right)^d
\]
uniformly on compact sets.  Therefore, we can consider the polynomials
\[
    R_{d,s}(\bm{x},\bm{y}) \, := \, \prod_{i=1}^m \. \prod_{j=1}^n \. \frac{1}{e^s}
    \left(\frac{sx_iy_j}{d} + 1\right)^d \quad \text{and} \quad
    \wt{R}_{d,s}(\bm{x},\bm{y}) \, := \,
    \prod_{i=1}^m \. \prod_{j=1}^n \. \frac{1}{e^s} \left(\frac{sx_i}{d} + y_j\right)^d.
\]
Since $\wt{R}_{d,s}$ is real stable, we can apply Theorem~\ref{thm:Lorentzian_bound} to get
\[
    \frac{[\wt{R}_{d,s}]_{\bm\alpha\,(md-\bm\beta)}}{\cpc_{\bm\alpha\,(md-\bm\beta)}(\wt{R}_{d,s})} \,
    \geq \,
    \prod_{i=2}^m \. \binom{nd}{\alpha_i} \. \frac{\alpha_i^{\alpha_i} (nd-\alpha_i)^{nd-\alpha_i}}{(nd)^{nd}}
    \prod_{j=1}^n \. \binom{md}{\beta_j} \. \frac{\beta_j^{\beta_j} (md-\beta_j)^{md-\beta_j}}{(md)^{md}}\,.
\]
Further, by Stirling's approximation we have
\[
    \lim_{d \to \infty} \.\binom{nd}{\alpha_i} \. \frac{\alpha_i^{\alpha_i} (nd-\alpha_i)^{nd-\alpha_i}}{(nd)^{nd}}
    \, = \, \frac{\alpha_i^{\alpha_i}}{\alpha_i!} \. e^{-\alpha_i}\ts,
\]
and the same holds for $\beta_j$. Combining this with the above bound gives
\[
    \lim_{d \to \infty} \, \frac{[\wt{R}_{d,s}]_{\bm\alpha\,(md-\bm\beta)}}{\cpc_{\bm\alpha\,(md-\bm\beta)}(\wt{R}_{d,s})}
    \, \geq \, e^{-2N+\alpha_1}\. \prod_{i=2}^m \. \frac{\alpha_i^{\alpha_i}}{\alpha_i!} \,
    \prod_{j=1}^n \. \frac{\beta_j^{\beta_j}}{\beta_j!}\,,
\]
where \ts $N := \sum_{i=1}^m \alpha_i = \sum_{j=1}^n \beta_j$.
Finally since \ts $R_{d,s} \to Q_{\infty,s}$ \ts coefficient-wise as \ts $d \to \infty$,
Lemma~\ref{lem:cap_lim_analytic} implies
\[
\begin{split}
    \frac{[Q_{\infty,s}]_{\bm\alpha\,\bm\beta}}{\cpc_{\bm\alpha\,\bm\beta}(Q_{\infty,s})}
    \, &= \, \lim_{d \to \infty} \, \frac{[R_{d,s}]_{\bm\alpha\,\bm\beta}}{\cpc_{\bm\alpha\,\bm\beta}(R_{d,s})}
    \, = \, \lim_{d \to \infty} \, \frac{[\wt{R}_{d,s}]_{\bm\alpha\,(md-\bm\beta)}}{\cpc_{\bm\alpha\,(md-\bm\beta)}(\wt{R}_{d,s})} \\
        &\geq \, e^{-2N+\alpha_1} \. \prod_{i=2}^m \. \frac{\alpha_i^{\alpha_i}}{\alpha_i!} \,
        \prod_{j=1}^n \. \frac{\beta_j^{\beta_j}}{\beta_j!} \, \geq \,
        e^{-2N} \. \prod_{i=1}^m \. \frac{\alpha_i^{\alpha_i}}{\alpha_i!} \. \prod_{j=1}^n \. \frac{\beta_j^{\beta_j}}{\beta_j!}\..
\end{split}
\]

We now compute the capacity of $Q_{\infty,s}$, which has an explicit formula due to the nice form of the function. Note first that we have
\[
    Q_{\infty,s}(\bm{x},\bm{y}) \, = \, \prod_{i=1}^m \. e^{-sn \ts + \ts s x_i  \sum_{j=1}^n y_j}\,.
\]
So the capacity expression can be broken up as follows:
\[
    \cpc_{\bm\alpha\,\bm\beta}(Q_{\infty,s}) \, = \, \inf_{\bm{y} > 0} \,
    \frac{1}{\bm{y}^{\bm\beta}} \. \prod_{i=1}^m \. \inf_{x_i > 0} \. \frac{e^{-sn \ts +\ts s x_i \sum_{j=1}^n y_j}}{x_i^{\alpha_i}}.
\]
For every \ts $i \in [m]$, note that
\[
    \frac{e^{-sn + s x_i \sum_{j=1}^n y_j}}{x_i^{\alpha_i}} \, = \,
    e^{-\alpha_i \log x_i - sn + s x_i \sum_{j=1}^n y_j}\..
\]
To minimize this expression, we only need to minimize the exponent. Applying calculus, we have
\[
    0 \, = \, \partial_{x_i} \left[-\alpha_i \log x_i - sn + s x_i \sum_{j=1}^n y_j\right] \, = \,
    -\frac{\alpha_i}{x_i} + s \sum_{j=1}^n y_j \implies x_i \, = \, \frac{\alpha_i}{s \sum_{j=1}^n y_j}\,.
\]
This gives
\[
    \inf_{x_i > 0} \. \frac{e^{-sn + s x_i \sum_{j=1}^n y_j}}{x_i^{\alpha_i}} = \frac{(se\sum_{j=1}^n y_j)^{\alpha_i}}{\alpha_i^{\alpha_i} e^{sn}} = \frac{(se)^{\alpha_i}}{\alpha_i^{\alpha_i} e^{sn}} \cdot \left(\sum_{j=1}^n y_j\right)^{\alpha_i},
\]
which in turn implies
\[
    \cpc_{\bm\alpha\,\bm\beta}(Q_{\infty,s}) \, = \,
    \frac{(se)^N}{\alpha^\alpha e^{smn}} \. \cdot \. \inf_{\bm{y} > 0} \.
    \frac{\left(\sum_{j=1}^n y_j\right)^N}{\bm{y}^{\bm\beta}}\,.
\]
By Lemma~\ref{lem:cap_linear}, we then have
\[
    \inf_{\bm{y} > 0} \, \frac{\left(\sum_{j=1}^n y_j\right)^N}{\bm{y}^{\bm\beta}} \, = \,
    \frac{N^N}{\bm\beta^{\bm\beta}}\,,
\]
which finally implies
\[
    \cpc_{\bm\alpha\,\bm\beta}(Q_{\infty,s}) \, = \,
    \frac{(seN)^N}{\bm\alpha^{\bm\alpha} \ts \bm\beta^{\bm\beta} \ts e^{smn}}\,.
\]
Combining everything then gives the desired bounds. \ $\sq$

\bigskip

\section{Comparing Bounds} \label{sec:compare_bounds}

In this section, we compare Barvinok's bounds to the bounds we are able to
achieve in this paper for counting contingency tables. To simplify the
computations, we use Stirling's approximation indiscriminately for every
factorial that appears. The approximation is in general only off by at most a factor
of \ts $\frac{e}{\sqrt{2\pi}}$, and it holds asymptotically as $\min\{\alpha_i,\beta_j\} \to \infty$ and $\min\{m,n\} \to \infty$.
This is the meaning in which we use \ts ``$\approx$'' \ts and \ts ``$\gtrsim$".

\medskip

\subsection{New bound vs.\ Barvinok's first bound}\label{ss:compare-1}
In~\cite{Bar09}, Barvinok achieves the following constant in the general case
\[
\begin{split}
    C_\mathrm{Barv}(K,\bm\alpha,\bm\beta) \, &= \,
    \frac{\Gamma(\frac{m+n}{2})}{2\ts e^5 \ts \pi^{\frac{m+n-2}{2}} \. mn\ts (N+mn)}\.
    \left(\frac{2}{(mn)^2 \ts (N+1)\ts (N+mn)}\right)^{m+n-1} \\
        &\qquad \times \frac{N!\. (N+mn)!\. (mn)^{mn}}{N^N \ts (N+mn)^{N+mn}\ts (mn)!} \,
        \prod_{i=1}^m \. \frac{\alpha_i^{\alpha_i}}{\alpha_i!} \, \prod_{j=1}^n \. \frac{\beta_j^{\beta_j}}{\beta_j!}\,,
\end{split}
\]
where \ts $N = \sum_i \alpha_i = \sum_j \beta_j$. We now compare to our constant:
\[
    C_\mathrm{new}(K,\bm\alpha,\bm\beta) \, = \,
    \prod_{i=2}^m \. \frac{\alpha_i^{\alpha_i}}{(\alpha_i+1)^{\alpha_i+1}} \,
    \prod_{j=1}^n \. \frac{\beta_j^{\beta_j}}{(\beta_j+1)^{\beta_j+1}}
    \, \approx \, \frac{1}{e^{m+n-1}} \, \prod_{i=2}^m \.\frac{1}{\alpha_i} \,
    \prod_{j=1}^n \.\frac{1}{\beta_j}\..
\]
First note that
\[
    \frac{N! \. (N+mn)!\. (mn)^{mn}}{N^N \ts (N+mn)^{N+mn} \ts (mn)!} \,
    \prod_{i=1}^m \. \frac{\alpha_i^{\alpha_i}}{\alpha_i!} \,
    \prod_{j=1}^n \. \frac{\beta_j^{\beta_j}}{\beta_j!} \, \approx \,
    \sqrt{\frac{N(N+mn)}{(2\ts \pi)^{m+n-1}\ts mn} \,
    \prod_{i=1}^m \. \frac{1}{\alpha_i} \, \prod_{j=1}^n \. \frac{1}{\beta_j}}\,.
\]
We then further have
\[
\begin{split}
    &\frac{\Gamma(\frac{m+n}{2})}{2\ts e^5 \ts \pi^{\frac{m+n-2}{2}} \. mn \ts (N+mn)}\left(\frac{2}{(mn)^2\ts (N+1)\ts (N+mn)}\right)^{m+n-1} \\
        &\qquad \approx \ \frac{\pi}{2\ts e^{\frac{11}{2}} \. mn\ts (N+mn)} \left(\frac{\sqrt{2(m+n)}}{\sqrt{e\pi}(mn)^2\ts (N+1)\ts (N+mn)}\right)^{m+n-1}\,.
\end{split}
\]
Combining these approximate equalities, gives
\[
    C_\mathrm{Barv}(K,\bm\alpha,\bm\beta) \ \approx \
    \left(\frac{\sqrt{m+n}}{\pi\ts \sqrt{e}\. (mn)^2\ts (N+1)\ts (N+mn)}\right)^{n+m-1} \,
    \sqrt{\frac{\pi^2 \ts N}{4\ts e^{11}\ts (mn)^3\ts (N+mn)} \, \prod_{i=1}^m \. \frac{1}{\alpha_i} \,
    \prod_{j=1}^n \. \frac{1}{\beta_j}}\,.
\]
We also have the more amenable bound
\[
    C_\mathrm{Barv}(K,\bm\alpha,\bm\beta) \, \lesssim \,
    \left(\frac{\sqrt{m+n}}{\pi\ts \sqrt{e} \. (mn)^2\ts (N+1)\ts (N+mn)}\right)^{n+m-1} \,
    \sqrt{\prod_{i=2}^m \. \frac{1}{\alpha_i} \, \prod_{j=1}^n \. \frac{1}{\beta_j}}\,.
\]
To compare to our constant $C_\mathrm{new}(K,\alpha,\beta)$ we use the easy bound
\[
    \prod_{i=2}^m \alpha_i \, \prod_{j=1}^n \beta_j \, \leq \, N^{m+n-1}\.,
\]
which leads to
\[
\begin{split}
    \frac{C_\mathrm{new}(K,\bm\alpha,\bm\beta)}{C_\mathrm{Barv}(K,\bm\alpha,\bm\beta)} \,
    &\gtrsim \, \left(\frac{\pi\ts (mn)^2\ts (N+1)\ts (N+mn)}{\sqrt{e\ts (m+n)}}\right)^{m+n-1} \,
    \sqrt{\prod_{i=2}^m  \. \frac{1}{\alpha_i} \, \prod_{j=1}^n \. \frac{1}{\beta_j}} \\
        &\gtrsim \, \left(N^{m+n-1}\right)^2 \sqrt{\prod_{i=2}^m \frac{1}{\alpha_i} \, \prod_{j=1}^n \. \frac{1}{\beta_j}}
        \, \ \gtrsim \, \ \left(N^{m+n-1}\right)^{\frac{3}{2}}\..
\end{split}
\]
That is, our lower bound (and approximation ratio) improves upon Barvinok's by at least the above factor.

\medskip

\subsection{New bound vs.\ Barvinok's second bound}\label{ss:compare-2}
%
%
%
There are now two features of this bound that we want to compare to ours:
the approximation ratio and the actual lower bound. For every valid~$K$,
we first note that
\begin{equation}\label{eq:H-P}
    H_N(\bm{x},\bm{y}) \. \leq \. P_K(\bm{x},\bm{y}) \quad \text{for all} \quad \bm{x},\bm{y} \. > \. 0,
\end{equation}
since $H_N$ and $P_K$ have the same coefficients on the support of $H_N$.
So in fact, if our approximation ratio is better than Barvinok's ration \ts $C_\mathrm{H}(\bm\alpha,\bm\beta)$,
then so is our lower bound.

To compare approximation ratios, we assume the $\beta_j$'s maximize \ts
$C_\mathrm{H}(\bm\alpha,\bm\beta)$ \ts and partially apply Stirling's approximation to get
\[
    C_\mathrm{H}(\bm\alpha,\bm\beta) \,\approx \,
    \binom{N+m-1}{m-1}^{-1} \binom{N+n-1}{n-1}^{-1} \left(\frac{1}{\sqrt{2\pi}}\right)^{n-1} \sqrt{\frac{N}{\prod_{j=1}^n \beta_j}}\,.
\]
This then gives
\[
    \frac{C_\mathrm{new}(\bm\alpha,\bm\beta)}{C_\mathrm{H}(\bm\alpha,\bm\beta)}
    \, \approx \, \binom{N+m-1}{m-1} \binom{N+n-1}{n-1} \.
    \frac{(\sqrt{2\pi})^{n-1}}{e^{m+n-1} \sqrt{N}} \,
    \prod_{i=2}^m \. \frac{1}{\alpha_1+1}\, \prod_{j=1}^n \.\frac{\sqrt{\beta_j}}{\beta_j+1}\,.
\]
We now make a few simple observations. First, if \ts $k \ll n$, by Stirling's approximation we have
\[
    \binom{n+k}{k} \, \approx \, \sqrt{\frac{n+k}{2\pi nk}} \cdot \frac{(n+k)^{n+k}}{k^k n^n}
    \, \approx \, \sqrt{\frac{n+k}{2\pi nk}} \left(\frac{e(n+k)}{k}\right)^k\ts.
\]
Then by the AM--GM inequality, we obtain:
\[
\begin{split}
    \prod_{j=1}^n \frac{1}{\beta_j+1} \, \geq \, \left(\frac{n}{N+n}\right)^n
    \, & \approx \, e^n \binom{N+n}{n}^{-1} \sqrt{\frac{N+n}{2\pi Nn}} \,
    = \, e^n \binom{N+n-1}{n-1}^{-1} \sqrt{\frac{n}{2\pi N(N+n)}}\,.
\end{split}
\]
Similarly,
\[
    \prod_{i=2}^m \frac{1}{\alpha_i+1} \, \geq \,
    \left(\frac{m-1}{N-\alpha_1+m-1}\right)^{m-1} \ \gtrsim \
    e^{m-1} \binom{N+m-1}{m-1}^{-1} \sqrt{\frac{N+m-1}{2\pi N(m-1)}}\,.
\]
Further, it is easy to see that
\[
    \frac{\sqrt{\beta_j}}{\beta_j + 1} \geq \frac{1}{\sqrt{2} \cdot \sqrt{\beta_j+1}}
    \quad \text{for} \quad \beta_j \.\geq \. 1,
\]
which implies
\[
\begin{split}
    \frac{C_\mathrm{new}(\bm\alpha,\bm\beta)}{C_\mathrm{H}(\bm\alpha,\bm\beta)} \
    &\gtrsim \
    \binom{N+n-1}{n-1}^{\frac{1}{2}} \left(\frac{\pi^{2n-5} n (N+m-1)^2}{e^{2n} 2^5 N^5 (N+n) (m-1)^2}\right)^{\frac{1}{4}} \\
        &\gtrsim \ \left(\frac{N+n-1}{n-1}\right)^{\frac{n-1}{2}} \left(\frac{\pi^{2n-6} n (N+m-1)^2 (N+n-1)}{e^2 2^6 N^6 (N+n) (m-1)^2 (n-1)}\right)^{\frac{1}{4}} \\
        &\gtrsim \ \frac{1}{\pi^3 N\sqrt{m}} \left(\frac{\pi(N+n-1)}{n-1}\right)^{\frac{n-1}{2}}.
\end{split}
\]
Therefore, our approximation ratio improves upon Barvinok's second approximation ratio.

\bigskip

\section{Volumes of flow and transportation polytopes} \label{sec:volume}

\subsection{The setup}\label{ss:volume-setup}
Let \ts $\bm{\alpha} = (\alpha_1,\ldots,\alpha_m) \in \N^m$ \ts and
\ts $\bm{\beta} = (\beta_1,\ldots,\beta_n) \in \N^n$ \ts be integer vectors.
A \emph{transportation polytope} \ts $\mathcal{T}_{\bm{\alpha},\bm{\beta}}$
\ts
is the set of  $m \times n$ real matrices $Z=(z_{ij})$, such that \ts $z_{ij}\ge 0$\ts,
and
\begin{equation}\label{eq:transp}
    \sum_{i=1}^m \. z_{ij} \. = \. \beta_j \text{\, for all \, $1\le j\le n\,,$}
    \quad \text{and} \quad \sum_{j=1}^n \. z_{ij} \. =\. \alpha_i
    \text{\, for all \, $1\le i\le m\ts.$}
\end{equation}
Transportation polytopes are classical objects of study in
geometric combinatorics and combinatorial optimization \cite{DK,EKK},
and their volume is one of the motivations to study contingency tables,
see e.g.~\cite{Bar09,CM09}.

The celebrated \emph{Birkhoff polytope} $\mathcal{B}_n$ is a special case
of the transportation polytope \ts $\mathcal{T}_{\bm\alpha,\bm\beta}$, when
$m=n$, and \ts $\bm\alpha=\bm\beta= (1,\ldots,1)\in \R^n$.  It is especially
well studied in its own right, see e.g.~\cite{EKK,Pak}.  First few
values of \ts $\Vol(\mathcal{B}_n)$ \ts are given in~\cite{CR,BP03,BP03b},
see also~\cite[\href{https://oeis.org/A037302}{A037302}]{OEIS}.  This is
also one of the few cases when the exact asymptotics for the volumes is
known, see Example~\ref{ex:volume-Birkhoff} below.

For a matrix $K=(k_{ij})$, a \emph{flow polytope} \ts
$\mathcal{F}_{K,\bm{\alpha},\bm{\beta}}$ \ts is defined by~\eqref{eq:transp}
and \ts $0\le z_{ij} \le k_{ij}$.  For $K=\infty$ we obtain the transportation
polytopes.  Note also that the $K$-contingency tables are integer
points in \ts $\mathcal{F}_{K,\bm{\alpha},\bm{\beta}}$.

The volume of flow polytopes has been actively studied in connection to
both discrete geometry and enumerative combinatorics. We refer to
\cite{BV,BDV,CDR}, and to more recent papers~\cite{B+,MM} for further references.

\medskip

\subsection{The results}\label{ss:volume-thm}
The connection between the number of $K$-contingency tables and the
volume of the corresponding flow polytope is given by the
following:
\[
    \Vol(\mathcal{F}_{K,\bm\alpha,\bm\beta}) \, = \,
    f(S,m,n) \. \cdot \. \lim_{M \to \infty} \.
    \frac{\CT_{MK}(M\bm\alpha,M\bm\beta)}{M^{(m-1)(n-1)}}\,,
\]
where $S=$ Supp$(K)$ and the \ts $f(S,m,n)^2$ \ts is
the covolume of the \ts lattice \ts $\Z\langle S\rangle
\cap \R\langle \mathcal{F}_{K,\bm\alpha,\bm\beta}\rangle$.

For the transportation polytopes, we have:
\[
    \Vol(\mathcal{T}_{\bm\alpha,\bm\beta}) \, = \,
    \sqrt{m^{n-1} \ts n^{m-1}} \. \cdot \. \lim_{M \to \infty} \.
    \frac{\CT(M\bm\alpha,M\bm\beta)}{M^{(m-1)(n-1)}}\,,
\]
where the covolume \ts $m^{n-1} \ts n^{m-1}$ \ts computed e.g.\
in~\cite[Lemma~3]{DE}.\footnote{The covolume \ts
$m^{n-1}\ts n^{m-1}$ \ts is equal to the number of spanning graphs in a complete
bipartite graph \ts $K_{mn}$, an observation which extends to all multigraphical
matrices~$K$, cf.~$\S$\ref{ss:finrem-hist}.}

\smallskip

\begin{theorem}[General lower bound]\label{t:volume-general-flow}
    Let $\bm\alpha \in \N^m$ and $\bm\beta \in \N^n$ be such that \ts
    $\sum_i \alpha_i = \sum_j \beta_j$. Let $K = (k_{ij})$ be an $m \times n$
    matrix with all entries $k_{ij} \in \N \cup \{+\infty\}$. Then:
    \[
        \Vol(\mathcal{F}_{K,\bm\alpha,\bm\beta}) \, \geq \,
        \frac{f(S,m,n)}{e^{m+n-1}} \,\.
        \prod_{i=2}^m\. \frac{1}{\alpha_i} \, \prod_{j=1}^n \.\frac{1}{\beta_j} \,\.
        \cpc_{\bm\alpha\,\bm\beta}\Biggl(\prod_{i=1}^m \. \prod_{j=1}^n \. \frac{(x_iy_j)^{k_{ij}}-1}{\log(x_i\ts y_j)}\Biggr),
    \]
    where $S=$ {\rm Supp}$(K)$ and $f(S,m,n)$ are as above.
\end{theorem}

For the transportation polytopes, we get:

\begin{theorem}\label{t:volume-general}
    Let $\bm\alpha \in \N^m$ and $\bm\beta \in \N^n$ be such that \ts
    $\sum_i \alpha_i = \sum_j \beta_j$. Then we have:
    \[
        \Vol(\mathcal{T}_{\bm\alpha,\bm\beta}) \, \geq \,
        \frac{\sqrt{m^{n-1} n^{m-1}}}{e^{m+n-1}} \,\.
        \prod_{i=2}^m\. \frac{1}{\alpha_i} \, \prod_{j=1}^n \.\frac{1}{\beta_j} \,\.
        \cpc_{\bm\alpha\,\bm\beta}\Biggl(\prod_{i=1}^m \. \prod_{j=1}^n \. \frac{-1}{\log(x_i\ts y_j)}\Biggr),
    \]
where the \ts $\inf$ \ts in the capacity is over \ts $0 < \bm{x}, \ts \bm{y} < 1$.
\end{theorem}

\smallskip

Note that \ts $\frac{(x_iy_j)^{k_{ij}}-1}{\log(x_iy_j)}$ \. becomes convex after you plug in
$e^x$ and $e^y$ and then take $\log$ on the outside.  Thus, one can easily
compute this capacity value using convex optimization as in the case of counting
contingency tables, see $\S$\ref{ss:finrem-alg}.

\smallskip

Before we present a proof, let us single out the case of uniform marginals which
are especially interesting and important in applications.

\smallskip

\begin{corollary}[Uniform marginals]\label{c:volume-uniform}
    For $\bm\alpha = (\alpha_0,\ldots,\alpha_0) \in \N^m$ and $\bm\beta = (\beta_0,\ldots,\beta_0) \in \N^n$, we have
    \[
        \Vol(\mathcal{T}_{\bm\alpha,\bm\beta})\, \geq \,
        \frac{(eN)^{(m-1)(n-1)}}{m^{(m-\frac{1}{2})(n-1)+1} \. n^{(n-\frac{1}{2})(m-1)}}\,.
    \]
\end{corollary}

\smallskip

Note that the results in~\cite{CM09} give the exact asymptotics only for uniform
marginals with \. $\max\bigl\{\frac{m}{n}, \. \frac{n}{m}\bigr\}=O(\log n)$, while the lower bound above
applies unconditionally.

\smallskip

\begin{example}\label{ex:volume-Birkhoff} {\rm
    For the Birkhoff polytope, the corollary gives:
\[
    \Vol(\mathcal{B}_n) \, \geq \, \frac{(en)^{(n-1)^2}}{n^{2(n-\frac{1}{2})(n-1)+1}}
    \, = \, e^{(n-1)^2} \ts n^{-(n^2-n+1)} \ts.
\]
This lower bound can be compared with the exact asymptotics given in~\cite{CM09}\ts{}:
$$
\Vol(\mathcal{B}_n) \, \sim \, C\cdot (2\pi)^{-n} \. e^{n^2+O(1)}\. n^{-(n-1)^2} \ts,
$$
for some known~$C>0$.
Note that our lower bound coincides with the actual asymptotic bound in the first two terms:
$$
\log \Vol(\mathcal{B}_n) \, = \, -n^2\log n \. + \. n^2 \. + \. O(n \log n).
$$
}
\end{example}

\medskip

\subsection{Proof of Theorem~\ref{t:volume-general-flow}}
By Corollary~\ref{c:main-CT}, we have:
\[
    \CT_{MK}(M\bm\alpha,M\bm\beta) \, \geq \, e^{1-m-n} \, \prod_{i=2}^m \. \frac{1}{M\alpha_i+1} \,
    \prod_{j=1}^n \. \frac{1}{M\beta_j+1} \, \cpc_{M\bm\alpha\,M\bm\beta}\Biggl(\prod_{i=1}^m \.\prod_{j=1}^n \. \frac{1-(x_i\ts y_j)^{Mk_{ij}+1}}{1-x_i\ts y_j}\Biggr).
\]
The first thing to note is that the constant in front of the capacity is asymptotically
\[
    e^{1-m-n} \, \prod_{i=2}^m \. \frac{1}{M\alpha_i+1} \,
    \prod_{j=1}^n \. \frac{1}{M\beta_j+1} \, \sim \, M^{1-m-n} \, e^{1-m-n} \,
    \prod_{i=2}^m \. \frac{1}{\alpha_i} \. \prod_{j=1}^n \. \frac{1}{\beta_j}\..
\]
Since the constant in the denominator of the volume limit expression is order \ts $M^{(m-1)(n-1)}$,
this means we need the order of the capacity term to be \ts $M^{mn}$. This is in fact the case,
and also we can get another capacity expression for the capacity term divided by \ts $M^{mn}$.
Specifically, note that
\[
    \frac{1}{M^{mn}} \, \inf_{\bm{x},\bm{y} > 0} \, \prod_{i=1}^m \. \prod_{j=1}^n \. \frac{(1-(x_i\ts y_j)^{Mk_{ij}+1})}{x_i^{M\alpha_i}\ts y_j^{M\beta_j} (1-x_iy_j)}
    \, = \, \inf_{\bm{x},\bm{y} > 0} \, \prod_{i=1}^m \. \prod_{j=1}^n \. \frac{1}{x_i^{\alpha_i} \ts y_j^{\beta_j}} \. \cdot \. \frac{(1-(x_i\ts y_j)^{k_{ij}+M^{-1}})}{M(1-(x_iy_j)^{M^{-1}})}\,.
\]
Next, we pass the limit on $M$ into the infimum in the capacity (swapping $\lim$ and $\inf$ is valid here by a standard argument, since we only need to prove a lower bound).
We then have:
\[
\begin{split}
    \lim_{M \to \infty} \, \frac{(1-(x_i\ts y_j)^{k_{ij}+M^{-1}})}{M(1-(x_iy_j)^{M^{-1}})} \, &= \,
    \lim_{M \to 0^+} \, \frac{M(1-e^{(k_{ij}+M) \log(x_i\ts y_j)})}{1-e^{M \log(x_i\ts y_j)}} \\
        &= \, \lim_{M \to 0^+} \, \frac{1 - (1 + M \log(x_i\ts y_j)) \. \cdot \. e^{(k_{ij}+M) \log(x_i\ts y_j)}}{-\log(x_i\ts y_j) \. \cdot \. e^{M \log(x_iy_j)}} \\
        &= \, \frac{(x_i\ts y_j)^{k_{ij}} - 1}{\log(x_i\ts y_j)}\..
\end{split}
\]
With this, we have
\[
    \Vol(\mathcal{F}_{K,\bm\alpha,\bm\beta}) \, \geq \,
    \frac{f(S,m,n)}{e^{m+n-1} \, \prod_{i=2}^m \ts \alpha_i \, \prod_{j=1}^n \ts \beta_j} \.
    \cpc_{\bm\alpha\,\bm\beta}\Biggl(\prod_{i=1}^m \. \prod_{j=1}^n \, \frac{(x_i\ts y_j)^{k_{ij}} - 1}{\log(x_i\ts y_j)}\Biggr),
\]
as desired.
 \ $\sq$

\medskip

\subsection{Proof of Theorem~\ref{t:volume-general}}

The same proof works here as was used above for Theorem~\ref{t:volume-general-flow}. The main difference is that we consider
\[
    \cpc_{M\bm\alpha\,M\bm\beta}\Biggl(\prod_{i=1}^m \.\prod_{j=1}^n \. \frac{1}{1-x_i\ts y_j}\Biggr),
\]
and so the infimum is over \ts $0 < \bm{x},\ts \bm{y} < 1$. Another way to see this is as a limit of the lower bound of Theorem~\ref{t:volume-general-flow} as \ts $K \to \infty$.

\medskip

\subsection{Proof of Corollary~\ref{c:volume-uniform}}

We can explicitly compute the capacity in this case. Consider
\[
    \log\left(\frac{\prod_{i=1}^m \. \prod_{j=1}^n \.
    \frac{-1}{\log(e^{x_i\ts +\ts y_j})}}{e^{\langle \bm{x}, \bm\alpha \rangle \ts +\ts \langle \bm{y}, \bm\beta \rangle}}\right)
    \, = \, -\langle \bm{x}, \bm\alpha \rangle \ts - \ts \langle \bm{y}, \bm\beta \rangle \. - \. \sum_{i=1}^m \. \sum_{j=1}^n \. \log(-x_i-y_j)\ts.
\]
The gradient of this expression at $x = y = -\frac{mn}{2N}$ where $N = m \cdot \alpha_0 = n \cdot \beta_0$ is then computed as
\[
  \Biggl(-\alpha_0 - \sum_{j=1}^n \frac{1}{x_1+y_j}\. , \. \ldots \. , \. -\beta_0 - \sum_{i=1}^m \frac{1}{x_i+y_1}\,,
    \. \ldots\. \Biggr)\Biggr|_{x \ts = \ts y \ts = \ts -\frac{mn}{2N}}
\]
So the above expression is minimized at \ts $x = y = -\frac{mn}{2N}$, which means the above capacity value is
\[
    \left.\frac{\prod_{i=1}^m \.\prod_{j=1}^n \. \frac{-1}{x_i+y_j}}{e^{\langle \bm{x}, \bm\alpha \rangle + \langle \bm{y},
    \bm\beta \rangle}}\right|_{x \ts =\ts y \ts = \ts -\frac{mn}{2N}} \, = \,
    \frac{(\frac{N}{mn})^{mn}}{e^{-mn}} \, = \, \left(\frac{eN}{mn}\right)^{mn}.
\]
Combining this with the above lower bound gives an explicit lower bound on the volume in the uniform case:
\[
    \Vol(\mathcal{T}_{\bm\alpha,\bm\beta}) \, \geq \,
    \frac{\sqrt{m^{n-1}\ts n^{m-1}}}{e^{m+n-1} \left(\frac{N}{m}\right)^{m-1} \left(\frac{N}{n}\right)^n} \,
    \left(\frac{eN}{mn}\right)^{mn} \, = \,
    \frac{(eN)^{(m-1)(n-1)}}{m^{(m-\frac{1}{2})(n-1)+1} \. n^{(n-\frac{1}{2})(m-1)}}\,,
\]
as desired.  \ $\sq$

\bigskip

\section{Uniform marginals}\label{sec:uniform}

In the case of uniform marginals, i.e., $\bm\alpha$ and $\bm\beta$ are both multiples
of the all-ones vector, we can explicitly compute the capacity value in the
upper and lower bounds of Theorem~\ref{thm:main_general_bound}. This then gives
the explicit upper and lower bounds for the number of contingency tables.

\smallskip

\subsection{Explicit upper and lower bounds} \label{ss:uniform-thm}
In this and the next section, we adopt the following convenient shorthand for bounds
on \ts $\CT(\bm{\alpha},\bm{\beta})$:

\medskip

{\small

\qquad \text{UB1} \ is the Barvinok first upper bound (see Theorem~\ref{thm:barvinok_general_bound}),

\vskip0.05cm

\qquad \text{LB1} \ is the Barvinok first lower bound (ibid.),

\vskip0.05cm

\qquad \text{UB2} \ is the Barvinok second upper bound (see Theorem~\ref{thm:barvinok_cs_bound}),

\vskip0.05cm

\qquad \text{LB2} \ is the Barvinok second lower bound (ibid.),

\vskip0.05cm

\qquad \text{UB3} \ is the Shapiro upper bound (see Remark~\ref{rem:main-shapiro}), and

\vskip0.05cm

\qquad {\text{New LB}} \ is our lower bound in the Main Theorem~\ref{thm:main_general_bound}.
}

\medskip

\begin{theorem}\label{t:uniform}
    Let \ts $\bm{\alpha} = (s,\ldots,s)\in \N^m$, \ts $\bm{\beta}=(t,\ldots,t)\in \N^n$, where \ts
    $m,n,s,t \in \N$, such that \ $ms=nt=N$, and $m\le n$. Then we have the following bounds on \ts
    $\CT(\bm{\alpha},\bm{\beta}):$\footnote{The LB1 is given only for $m+n\ge 10$. }
$$
\aligned
\text{\rm UB1}  \ & = \  \frac{(N+mn)^{N+mn}}{N^N \ts (mn)^{mn}}\\
\text{\rm LB1} \ & = \
\frac{2^{m+n-2}\. \Gamma(\frac{m+n}{2}) \. N! \, (N+mn)! \ s^{sm}\. t^{tn}}
            {e^5 \. \pi^{\frac{m+n-2}{2}} \. (N+mn)^{m+n} \. (N+1)^{m+n-1}\. N^{2N}  \. (mn)^{2m+2n-1}\. (mn)! \ (s!)^{m} \. (t!)^{n}} \\
\text{\rm UB2}  \ & = \  \binom{N+mn-1}{N}\\
\text{\rm LB2} \ & = \ \text{\rm UB2} \ \binom{N+m-1}{m-1}^{-1} \binom{N+n-1}{n-1}^{-1} \.
        \frac{N!\. s^{sm}}{N^N \. (s!)^m} \\
\text{\rm UB3}  \ & = \  \text{\rm UB1} \,\. \frac{1}{\left(1+\frac{N}{mn}\right)^{m+n-1}} \ \.
= \ \. \frac{(N+mn)^{N+(m-1)(n-1)}}{N^N (mn)^{(m-1)(n-1)}}\\
\text{\rm New LB} \ & = \ \text{\rm UB1} \,\. \frac{s^{s(m-1)}\ts t^{tn}}{(s+1)^{(s+1)(m-1)}\ts (t+1)^{(t+1)n}}
\endaligned
$$
\end{theorem}

\smallskip

Note that \. UB1 $>$ UB2 \. in this case.  This a general fact which follows
from~\eqref{eq:H-P}.  Note also that we trivially have \.  UB2 $> \CT(\bm{\alpha},\bm{\beta})$, since
\ts UB2 \ts counts \emph{all} tables with sum~$N$ irrespective of the row/column
constraints.  That makes only the lower bounds nontrivial in this case, and possibly UB3 when
$s$ and~$t$ are large.  This is confirmed by the numerical results in the next section.

\medskip

\subsection{Capacity calculations}\label{ss:uniform-calculations}
Recall the notation
\[
    P_\infty(\bm{x}, \bm{y}) \, := \, \prod_{i=1}^m \prod_{j=1}^n (1-x_iy_j)^{-1}.
\]
The explicit computation of the capacity of $P_\infty$ in the uniform case is then given by
following result.

\smallskip

\begin{lemma}\label{l:Cap1-uniform}
    Let \ts $\bm{\alpha} = (s,\ldots,s)\in \N^m$, \ts $\bm{\beta}=(t,\ldots,t)\in \N^n$, where \ts
    $m,n,s,t \in \N$, such that \ $ms=nt=N$. Then we have:
    \[
        \cpc_{\bm\alpha\,\bm\beta}(P_\infty) \, = \,  \frac{(N+mn)^{N+mn}}{N^N (mn)^{mn}}\,.
    \]
\end{lemma}

\begin{proof}  In~\cite{Bar12,Bar17}, using prior work and duality, Barvinok shows (in greater generality)
that the capacity bound is equal to $\exp g(Z)$, where  $Z=(z_{ij})$ is the
unique maximum of the strictly convex function
$$
g(Z) \, := \, \sum_{i=1}^{m}\. \sum_{j=1}^{n} \. (z_{ij}+1) \log (z_{ij}+1) \. - \.
z_{ij} \log z_{ij}\ts,
$$
and the maximum is over the transportation polytope \ts $\mathcal{T}_{\alpha,\beta}$.
By the symmetry, this unique maximum is attained at \ts $z_{ij}=N/mn$, and we have:
$$
g(Z) \, = \, mn \ts \Biggl(\frac{N}{mn}+1\Biggr) \. \log \Biggl(\frac{N}{mn}+1\Biggr)
\, - \, mn \.\frac{N}{mn} \. \log \frac{N}{mn}\..
$$
Therefore,
$$
\cpc_{\bm\alpha\,\bm\beta}(P_\infty) \, = \, e^{g(Z)}\, = \,  \frac{(N+mn)^{N+mn}}{N^N (mn)^{mn}}\,,
$$
as desired.
\end{proof}

\smallskip

Additionally, we need to be able to compute the capacity of the complete symmetric polynomials,
used in Theorem~\ref{thm:barvinok_cs_bound}. Recall from Section~\ref{ss:compare-2} the notation
\[
    H_N(\bm{x},\bm{y}) \, = \, \sum_{K} \.
    \prod_{i=1}^m \. \prod_{j=1}^n \. (x_iy_j)^{k_{ij}} \, = \, h_N(\bm{x}\cdot\bm{y}),
\]
where the sum is over all \ts $K=(k_{ij})$ \ts with total sum $N$ of the entries:
\ts $\sum_{i,j} k_{ij} = N$, and $h_N$ is the complete homogeneous symmetric polynomial of
degree $N$ in $mn$ variables \ts $x_iy_j$. The explicit computation of the capacity of $H_N$
in the uniform case is then given as follows.

\smallskip

\begin{lemma}\label{l:Cap2-uniform}
    Let \ts $\bm{\alpha} = (s,\ldots,s)\in \N^m$, \ts $\bm{\beta}=(t,\ldots,t)\in \N^n$, where \ts
    $m,n,s,t \in \N$, such that \ $ms=nt=N$. Then we have:
    \[
        \cpc_{\bm\alpha\,\bm\beta}(H_N) \, = \, H_N(\bm{1},\bm{1}) \, = \, \binom{N+mn-1}{N}.
    \]
\end{lemma}

\begin{proof}
    The second equality comes from the fact that the complete symmetric polynomial of
    degree $N$ in $mn$ variables, evaluated at the all-ones vector, counts the number
    of degree $N$ monomials in $mn$ variables. For the first equality, note that by symmetry we have
    \[
        \partial_{x_1} H_N(\bm{1},\bm{1}) \, = \,
        \frac{1}{m} \. \sum_{\ell=1}^m \. \partial_{x_\ell} H_N(\bm{1},\bm{1})
        \, = \, \frac{1}{m} \. \sum_{\ell=1}^m \. \sum_{\sum_{i,j} \ts k_{ij} = N} \.
        \sum_{j=1}^n k_{\ell j} \, = \, \frac{N}{m} \. H_N(\bm{1},\bm{1}),
    \]
    and similarly,
    \[
        \partial_{y_1} H_N(\bm{1},\bm{1}) \, = \frac{N}{n} \. H_N(\bm{1},\bm{1}).
    \]
    Therefore, we in fact have
    \[
        \left.\nabla \log(H_N)\right|_{\bm{x}=\bm{1},\bm{y}=\bm{1}} \, = \,
        \left(\ts\frac{N}{m}\., \. \ldots \. , \. \frac{N}{m}\., \.\frac{N}{n}\., \.\ldots\., \.\frac{N}{n}\ts\right) \,
        = \, (s, \ldots,s, \. t,\ldots, t).
    \]
    By Fact~2.10 of~\cite{GL18}, this implies \ts $\cpc_{\bm\alpha\,\bm\beta}(H_N) = H_N(\bm{1},\bm{1})$.
\end{proof}

\medskip

\begin{proof}[Proof of Theorem~\ref{t:uniform}] The exact values of UB1 and UB2
are given by the lemmas above.  The formulas for LB1, LB2 and New~LB follow immediately
from Theorem~\ref{thm:barvinok_general_bound}, Theorem~\ref{thm:barvinok_cs_bound} and
Main Theorem~\ref{thm:main_general_bound}, respectively.  For LB2, we use the fact
that $s\ge t$. Finally, the Shapiro correction term simplifies in the uniform case.
Indeed, the  product in the minimum the Remark~\ref{rem:main-shapiro} is equal on all spanning
trees $\tau\in K_{mn}$, which all have \ts $(m+n-1)$ \ts edges, and all edges have
the same weight \ts $\frac{1}{1+z_{ij}}\ts =\ts \frac{1}{1+N/mn}$ \ts from the proof of Lemma~\ref{l:Cap1-uniform}.
We omit the details.
\end{proof}

\bigskip

\section{Numerical examples}\label{sec:numerical}

In this section we compare the bounds numerically in specific cases where
the exact number of contingency tables is known either exactly or approximately.
Such comparisons are particulary easy when the marginals are uniform,
see Theorem~\ref{t:uniform}.

\medskip

\subsection{Uniform marginals}\label{ss:numerical-uniform}
Our first table consists of comparisons in the case of uniform marginals.
We give bounds for the number \ts $\CT(\bm\alpha,\bm\beta)$ \ts of
$m\times n$ contingency tables with row sums~$s$ and column sums~$t$,
so \ts $N= m\ts s = n \ts t$.  We use the notation in~$\S$\ref{ss:uniform-thm}
and explicit formulas from Theorem~\ref{t:uniform}.

\bigskip

\noindent
\resizebox{\columnwidth}{!}{
\begin{tabular}{c|cc|cc|ccc|c|ccc}
\textbf{Case} & $m$ & $n$ & $s$ & $t$ & UB1 & UB2 & UB3 & \color{blue}{\bf Actual} & \color{red}{\bf New LB} & LB2 & LB1 \\
        \hline\hline
\textbf{1}& 3 & 3 & 100 & 100 & $4.7 \times 10^{17}$ & $1.8 \times 10^{15}$ & $3.4 \times 10^{11}$ & \color{blue}{$1.3 \times 10^{7}$} & \color{red}{$3.1 \times 10^{5}$} & $2.4 \times 10^{3}$ & $1.5 \times 10^{-28}$ \\

\textbf{2} & 3 & 9 & 99 & 33 & $2.3 \times 10^{40}$ & $1.5 \times 10^{38}$ & $3.7 \times 10^{29}$ & \color{blue}{$2.8 \times 10^{21}$} & \color{red}{$7.3 \times 10^{17}$} & $5.6 \times 10^{15}$ & $1.2 \times 10^{-62}$ \\

\textbf{3} & 3 & 49 & 98 & 6 & $8.1 \times 10^{121}$ & $1.1 \times 10^{120}$ & $1.1 \times 10^{98}$ & \color{blue}{$1.0 \times 10^{68}$} & \color{red}{$9.1 \times 10^{55}$} & $6.4 \times 10^{53}$ & $4.1 \times 10^{-381}$ \\

\textbf{4}& 10 & 10 & 20 & 20 & $8.5 \times 10^{82}$ & $1.4 \times 10^{81}$ & $2.2 \times 10^{74}$ & \color{blue}{$1.1 \times 10^{59}$} & \color{red}{$5.7 \times 10^{49}$} & $4.8 \times 10^{41}$ & $5.2 \times 10^{-104}$ \\

\textbf{5}& 18 & 18 & 13 & 13 & $6.4 \times 10^{164}$ & $1.3 \times 10^{163}$ & $6.0 \times 10^{156}$ & \color{blue}{$7.9 \times 10^{127}$} & \color{red}{$1.1 \times 10^{110}$} & $2.7 \times 10^{95}$ & $1.1 \times 10^{-214}$ \\

\textbf{6}& 30 & 30 & 3 & 3 & $9.5 \times 10^{130}$ & $3.8 \times 10^{129}$ & $3.8 \times 10^{128}$ & \color{blue}{$2.2 \times 10^{92}$} & \color{red}{$2.2 \times 10^{73}$} & $1.6 \times 10^{56}$ & $2.2 \times 10^{-522}$ \\

\textbf{7}& 100 & 100 & 3 & 3 & $1.2 \times 10^{589}$ & $2.8 \times 10^{587}$ & $3.4 \times 10^{586}$ & \color{blue}{$5.3 \times 10^{459}$} & \color{red}{$4.9 \times 10^{394}$} & $4.1 \times 10^{332}$ & $1.5 \times 10^{-2267}$ \\
\hline



\textbf{8}& 4 & 4 & 300 & 300 & $9.9 \times 10^{36}$ & $1.3 \times 10^{34}$ & $5.1 \times 10^{25}$ & \color{blue}{$2.0 \times 10^{19}$} & \color{red}{$4.1 \times 10^{16}$} & $3.8 \times 10^{12}$ & $2.5 \times 10^{-39}$ \\

\textbf{9}& 9 & 9 & $10^3$ & $10^3$ & $1.1 \times 10^{201}$ & $4.4 \times 10^{197}$ & $1.8 \times 10^{168}$ & \color{blue}{$8.0 \times 10^{151}$} & \color{red}{$4.5 \times 10^{142}$} & $7.3 \times 10^{128}$ & $1.8 \times 10^{-32}$ \\

\textbf{10}& 9 & 9 & $10^5$ & $10^5$ & $7.7 \times 10^{362}$ & $3.1 \times 10^{357}$ & $1.4 \times 10^{298}$ & \color{blue}{$6.1 \times 10^{279}$} & \color{red}{$3.2 \times 10^{270}$} & $5.2 \times 10^{248}$ & $1.5 \times 10^{44}$ \\

\textbf{11} & 15 & 15 & $10^3$ & $10^3$ & $6.7 \times 10^{508}$ & $2.6 \times 10^{505}$ & $3.8 \times 10^{457}$ & \color{blue}{$\approx 1.7 \times 10^{427}$} & \color{red}{$1.7 \times 10^{409}$} & $2.3 \times 10^{384}$ & $1.3 \times 10^{80}$ \\

\textbf{12}& 15 & 15 & $10^5$ & $10^5$ & $1.3 \times 10^{958}$ & $5.1 \times 10^{952}$ & $1.1 \times 10^{851}$ & \color{blue}{$\approx 1.7 \times 10^{819}$} & \color{red}{$3.2 \times 10^{800}$} & $4.5 \times 10^{761}$ & $4.0 \times 10^{383}$ \\

\textbf{13}& 100 & 100 & $10^3$ & $10^3$ & $1.3 \times 10^{14553}$ & $6.0 \times 10^{14549}$ &  $8.2 \times 10^{14346}$ & \color{blue}{$\approx 6.3 \times 10^{14072}$} & \color{red}{$5.3 \times 10^{13869}$} & $4.6 \times 10^{13684}$ & $5.0 \times 10^{10741}$ \\

\textbf{14}& 100 & 100 & $10^5$ & $10^5$ & $1.3 \times 10^{34345}$ & $5.2 \times 10^{34339}$ & $1.1 \times 10^{33751}$ & \color{blue}{$\approx 6.3 \times 10^{33470}$} & \color{red}{$4.9 \times 10^{33263}$} & $4.4 \times 10^{32979}$ & $6.2 \times 10^{29545}$ \\
\hline\hline    \end{tabular}
}

\bigskip

{\small

\noindent
Here the actual values in cases 1--6 are taken from \cite[Table~1]{CM10},
in case~7 from~\cite[\href{https://oeis.org/A001500}{A001500}]{OEIS}
(computed by Heinz),
and in case~8 is from~\cite[Table~3]{dedetails} (see also~\cite[p.~27]{Lan}).
Actual values in cases 9--10 are computed from the exact form of the \emph{Ehrhart
polynomial} for the Birkhoff polytope $\mathcal{B}_9$ given in~\cite{BP03b}.

In the last four cases 11--14, the number of tables
is only given approximately and likely imprecise, but giving the
right order of magnitude.  In cases 11--12, we used a numerical
approximation for the volume of the Birkhoff polytope $\mathcal{B}_{15}$
given in~\cite[Table 6]{CV16} (see also~\cite[Table 1]{EF}). Finally, in cases 13--14,
we used the exact asymptotics, given in~\cite[Thm~1]{CM10}.}

\medskip

\subsection{Non-uniform marginals}\label{ss:numerical-non-uniform}
In the table below, the last column ``Time'' gives the CPU time it
took to compute Barvinok's UB2 and LB2. To compute UB1, LB1,
and our lower bound, the time never exceeds 2 seconds.
The stark difference between these two cases comes from
the fact that the complete symmetric polynomials associated to UB2/LB2
(see Section~\ref{ss:compare-2}) take much longer to
compute than the rational function
\ts $P_\infty(\bm{x}, \bm{y}) = \prod_{ij} (1-x_iy_j)^{-1}$ \ts
(see, however~$\S$\ref{ss:finrem-alg}).

\bigskip

\noindent
\resizebox{\columnwidth}{!}{
    \begin{tabular}{c|ccc|ccc|c|ccc|c}
        \textbf{Case} & m & n & N & UB1 & UB2 & UB3 & \color{blue}{\bf Actual} & \color{red}{\bf New LB} & LB2 & LB1 & Time \\
        \hline\hline
       \textbf{1} & 4 & 4 & 592 & $3.0 \times 10^{30}$ & $6.0 \times 10^{27}$ & $7.1 \times 10^{18}$ & \color{blue}{$1.2 \times 10^{15}$} & \color{red}{$9.5 \times 10^{12}$} & $4.6 \times 10^{8}$ & $3.8 \times 10^{-40}$ & 79 sec \\

       \textbf{2} & 5 & 4 & 1269 & $1.4 \times 10^{34}$ & $1.2 \times 10^{31}$ & $8.3 \times 10^{20}$ & \color{blue}{$3.4 \times 10^{16}$} & \color{red}{$2.0 \times 10^{14}$} & $3.0 \times 10^{7}$ & $1.5 \times 10^{-52}$ & 550 sec \\

        \textbf{3} & 4 & 4 & 65159458 & $1.3 \times 10^{112}$ & ? & $2.1 \times 10^{65}$ & \color{blue}{$4.3 \times 10^{61}$} & \color{red}{$5.8 \times 10^{58}$} & ? & $2.3 \times 10^{-49}$ & N/A \\

        \textbf{4} & 50 & 50 & 486 & $7.2 \times 10^{562}$ & ? & $1.3 \times 10^{551}$ & \color{blue}{??} & \color{red}{$5.2 \times 10^{421}$} & ? & $6.4 \times 10^{-749}$ & N/A \\

        \textbf{5} & 50 & 50 & 302 & $1.2 \times 10^{350}$ & ? & $7.3 \times 10^{338}$ & \color{blue}{??} & \color{red}{$1.1 \times 10^{239}$} & ? & $2.0 \times 10^{-922}$ & N/A \\
    \hline\hline\end{tabular}
}

\bigskip


{\small

\noindent
In the above table, case 1 is a celebrated example given by
$\bm{\alpha} = (220, 215, 93, 64)$ and
$\bm{\beta} = (108, 286, 71, 127)$.  It was first studied
in~\cite{DE}, the exact value was reported in~\cite[$\S$6]{DG95},
and further discussed as a benchmark in \cite{Bar17,BH12,D09,DS}.
Case~2 is given by
\ts $\bm{\alpha} = (9, 49, 182, 478, 551)$ and $\bm{\beta} = (9, 309, 355, 596)$,
was studied in~\cite[$\S$6.4]{C+05}.
Case~3 is a large example computed in~\cite[Table~3]{dedetails}, and is given by
{\small
$$\bm{\alpha} \. = \. (13070380, 18156451, 13365203, 20567424), \quad
\bm{\beta} \. = \. (12268303, 20733257, 17743591, 14414307).
$$
}
Case~4 is given by
\vspace{-0.16cm}
\begin{equation*}
\resizebox{.95\hsize}{!}{
$
    \bm{\alpha} \. = \. (10, 8, 11, 11, 13, 11, 10, 9, 7, 9, 10, 16, 11, 9,
    12, 14, 12, 7, 9, 10, 10, 6, 11, 8, 9, 8, 14, 12,
    5, 10, 10, 8, 7, 8, 10, 10, 14, 6, 10, 7, 13, 4, 6,
    8, 9, 15, 11, 12, 10, 6),
$
}
\end{equation*}
\vspace{-1.8em}
\begin{equation*}
\resizebox{.95\hsize}{!}{
$
    \bm{\beta} \. = \. (9, 6, 12, 11, 9, 8, 8, 11, 9, 11, 13, 7, 10, 8, 9, 7,
    8, 3, 10, 11, 13, 7, 5, 11, 10, 9, 10, 13, 9, 9, 7, 7,
    6, 8, 10, 12, 8, 12, 16, 12, 15, 12, 13, 13, 10, 7,
    12, 13, 6, 11).
$
}
\end{equation*}
\noindent
These marginals were introduced in~\cite[$\S$6.1]{C+05}, where an estimate
\ts $\BCT(\bm{\alpha},\bm{\beta})=(7.7 \pm .1) \times 10^{432}$ \ts was
given for the number of binary contingency tables using the sequential
importance sampling~(SIS). For comparison, Gurvits's Theorem~\ref{thm:main_binary_bound}
gives a large upper bound \ts $1.3 \times 10^{515}$, but a very sharp
lower bound \ts $8.9 \times 10^{431}$, with a $0.5$~sec.\ CPU time.

Finally, case~5 is given by
\vspace{-0.16cm}
\begin{equation*}
\resizebox{.95\hsize}{!}{
$
    \bm{\alpha} \. = \. (14, 14, 19, 18, 11, 12, 12, 10, 13, 16, 8, 12, 6, 15, 6, 7, 12, 1, 12, 3, 8, 5, 9, 4, 2, 4, 1, 4, 4, 5, 2, 3, 3, 1, 1, 1, 2, 1, 1, 2, 1, 3, 3, 1, 3, 2, 1, 1, 1, 2),
$
}
\end{equation*}
\vspace{-1.8em}
\begin{equation*}
\resizebox{.95\hsize}{!}{
$
    \bm{\beta} \. = \. (14, 13, 14, 13, 13, 12, 14, 8, 11, 9, 10, 8, 9, 8, 4, 7, 10, 9, 6, 7, 6, 5, 6, 8, 1, 6, 6, 3, 2, 3, 5, 4, 5, 2, 2, 2, 3, 2, 4, 3, 1, 1, 1, 3, 2, 2, 3, 5, 2, 5).
$
}
\end{equation*}

\noindent
These marginals were also introduced in~\cite[$\S$6.1]{C+05}, with
an estimate \ts $\BCT(\bm{\alpha},\bm{\beta})=(8.78 \pm .05) \times 10^{242}$.
Here Gurvits's Theorem~\ref{thm:main_binary_bound}
gives an upper bound \ts $1.7 \times  10^{309}$, and a  sharp
lower bound \ts $3.0 \times 10^{240}$, again with a $0.5$~sec.\ CPU time.

Note that we were unable to finish computation of the UB2 and LB2 in cases 3--5,
which overwhelmed our computer system.  This could be a problem with our
implementation, of course.  We would be curious to see these bounds if
someone could compute them.

}
\medskip

\subsection{Discussion}  \label{ss:numerical-discussion}
All bounds above are best viewed on the log-scale,
since they are multiplicative in nature and the approximation ratios grow
exponentially otherwise.  However, as \ts $\min\{\alpha_i,\beta_j\}\to \infty$,
we have \ts $\log$s of all bounds equal to \ts
$(1+o(1)) \log \CT(\bm\alpha,\bm\beta)$, even if the rate of
convergence implied by $o(1)$ notation vary greatly between the bounds.

Now, in all examples we computed, New~LB dominates
LB2 and dwarfs LB1, confirming the asymptotic bounds in
Section~\ref{sec:compare_bounds}. In fact, the latter is smaller than~1 in many cases.
In addition, as discussed above, New~LB is much faster to compute than
the LB2, sometimes by orders of magnitude faster.  Furthermore, in many
examples, especially with non-uniform margins, the upper bounds are
rather far from the actual number of contingency tables, while our
New~LB is much closer (on a log-scale).

We should mention an unusual situation in cases~4 and~5, when New~LB for
\ts $\CT(\bm{\alpha},\bm{\beta})$ \ts is smaller than Gurvits's~LB
for \ts $\BCT(\bm{\alpha},\bm{\beta})$. This is very counterintuitive,
and suggests that for relatively small marginals sometimes taking smaller
$K=(k_{ij})$ can give greater lower bounds for \ts $\CT_K(\bm{\alpha},\bm{\beta})$
than taking $K=\infty$ gives for \ts $\CT(\bm{\alpha},\bm{\beta})$.

Although we concentrate our efforts on the lower bounds, let us make
some observations about the upper bounds.  It is clear from definition
that UB1 $>$ UB3, but the relationship of UB2~vs.~UB3 is not so clear.
In fact, for $s$~and~$t$  small relative to $m$~and~$n$, the Shapiro
correction term is very small, and UB3 become close to~UB1.

\bigskip

\section{Final remarks}\label{sec:finrem}

\subsection{}\label{ss:finrem-hist}
There are many variations on the problem of counting general (unrestricted)
and binary contingency tables.  These include symmetric tables with
zero diagonal, which correspond to graphs with fixed degrees,
see~\cite{Wor}.  The technique of typical matrices was extended
to this setting in~\cite{BH13}.  High-dimensional tables are especially
important in statistical applications~\cite{Ev,Kat}, but even harder
to analyze computationally~\cite{DO}.  We refer
to~\cite{Ben} for a recent extension of Barvinok's first lower
bound to this setting.  It would be interesting if the Lorentzian
polynomials technique can be extended or modified in either of
these two directions.

In the paper, we consider only a special case of flow polytopes
and integer flows corresponding to weighted bipartite graphs,
see e.g.~\cite{BV,B+}.  In fact, flow polytopes for general
directed graphs can be reduced to this case via a simple
\emph{BDV--transformation} of graphs, which gives a bijection
between the flows~\cite{BDV}.  Finally, our lower bound can be
further extended to \emph{weighted contingency tables}, which can be
viewed as evaluations of the natural generating function of
$\CT(\bm{\alpha},\bm{\beta})$, see e.g.~\cite[$\S$8.5]{Bar16}.

\smallskip

\subsection{}\label{ss:finrem-alg}
From the computational complexity point of view, the upper
and lower bounds in our paper can be viewed as a deterministic
approximation algorithm. The algorithm has an exponential
 approximation ratio, of course.  For some special cases
 of the problem, such as the permanent, there is a probabilistic
 polynomial time approximation algorithm~\cite{JSV}.

Now, capacities in the paper are solutions of a convex polynomial
optimization problem, which can be solved in polynomial time by
the classical interior point methods, see e.g.~\cite{NN}. The minimization of
Shapiro's correction term in Remark~\ref{rem:main-shapiro}
is an instance of the minimum spanning tree problem and
can be solved by the greedy algorithm.  On the other hand,
computing the complete homogeneous polynomial necessary for Barvinok's
second bound (Theorem~\ref{thm:barvinok_cs_bound}) is harder, but can
be done efficiently in several different ways, e.g.\ via~\cite[Thm~1]{BIM}.

For flow polytopes, the covolume $f(S,m,n)$ in~$\S$\ref{ss:volume-thm}
is a matrix determinant of size at most~$mn$, and thus easy to compute.
Finally, the \emph{BDV--transformation} mentioned above gives at most a
quadratic blowup in the size of the graph, and is also easy to compute.

\smallskip

\subsection{}\label{ss:finrem-graphs}
We can interpret a graphical matrix $K$ as an $m \times n$ adjacency matrix of
a bipartite graph~$G$.  Then \ts $\CT_K(\alpha,\beta)$ \ts is the number of
subgraphs of $G$ with degree sequences given by $\bm{\alpha}$ and $\bm{\beta}$.
In this case, the vectors $\bm{\lambda}$ and $\bm{\gamma}$ in
Theorem~\ref{thm:main_binary_bound} correspond to the degree
sequences of~$G$.

\smallskip

\subsection{}\label{ss:finrem-indep} There is a surprisingly strong
(unproven) \emph{independence heuristic} due to Good~\cite{Good},
for the number of (unconstrained) contingency tables:
$$
\CT(\bm{\alpha}, \bm{\beta}) \, \approx \, \CTI(\bm{\alpha}, \bm{\beta}) \, := \, \binom{N+mn-1}{mn-1}^{-1} \.
\prod_{i=1}^{m} \binom{\alpha_i+n-1}{n-1} \, \prod_{j=1}^{n} \binom{\beta_j+m-1}{m-1}\ts.
$$
This heuristic is discussed further in~\cite{Bar09,DE,DG95}.  For the uniform marginals,
a weak version of the heuristic is stated as a conjecture in~\cite[Conj.~1]{CM10},
where it is proved asymptotically in some ``near-square'' cases.

For example, in the uniform case~4 in~$\S$\ref{ss:numerical-uniform},
we have \ts $\CTI(\bm{\alpha}, \bm{\beta}) = 7.4 \times 10^{58}$,
which is much closer to the actual value \ts $\CT(\bm{\alpha}, \bm{\beta}) = 1.1\times 10^{59}$ \ts
than any of the bounds. Similarly, in the non-uniform case~3 in~$\S$\ref{ss:numerical-non-uniform},
we have \ts $\CTI(\bm{\alpha}, \bm{\beta}) = 3.7 \times 10^{61}$,
which is again much closer to the actual value \ts $\CT(\bm{\alpha}, \bm{\beta}) = 4.3\times 10^{61}$ \ts
than any of the bounds.  Finally, for the non-uniform case~4 in~$\S$\ref{ss:numerical-non-uniform},
we have \ts $\CTI(\bm{\alpha}, \bm{\beta}) = 7.8\times 10^{471}$, a very reasonable guess
given that the New~LB is clearly undercounting \ts $\CT(\bm{\alpha}, \bm{\beta})$ \ts in this case,
cf.~$\S$\ref{ss:numerical-discussion}.

\smallskip

\subsection{}\label{ss:finrem-special}
Let $m$, $n$ and $N$ be fixed, and let $K=\infty$.  When we vary the marginals
\ts $\bm{\alpha}$ \ts and \ts $\bm{\beta}$ \ts over all partitions of~$N$,
the uniform case has the largest approximation ratio in the
Main Theorem~\ref{thm:main_general_bound}, while $\CT(\bm{\alpha},\bm{\beta})$
is also the largest of all such marginals, see~\cite{Bar10b,PP}.
This explains why our New~LB can be still far away from the actual value
in~$\S$\ref{ss:numerical-uniform}.  This can also be seen in the approximation
ratio in Theorem~\ref{thm:almost_const} and in the lower term gap in the
volume of the Birkhoff polytope (Example~\ref{ex:volume-Birkhoff}).

On the other hand, all previously known lower bounds and other techniques
tend to behave rather poorly when the matrix is far from uniform, and this
includes the MCMC algorithms, see \cite{BLSY,DKM}.  So perhaps our lower
bound coupled with Shapiro's upper bound are the only provably good
bounds in that case.

Let us also mention that it is unlikely there is a universally good
lower and upper bound for the general \ts $\CT(\bm{\alpha},\bm{\beta})$.
Sidestepping conjectural hardness of approximation results in computational
complexity, there is also a probabilistic evidence of this phenomenon.
In the simplest nonuniform case with marginals of two types, we already
have a phase transition for the number of contingency tables, first
predicted in~\cite{Bar10b}, and recently proved in~\cite{DLP}.

\vskip.3cm

{\small
\subsection*{Acknowledgements}
We are grateful to Nima Anari, Sasha Barvinok, Sam Dittmer, Leonid Gurvits,
Mark Handcock, Han Lyu, Brendan McKay and Greta Panova for many
helpful discussions.  Special thanks to Sasha Barvinok for telling
us about \cite[$\S$8]{Bar16} and~\cite{Sha}, to Jes\'us De Loera for
showing us~\cite{Lan}, and to Matt Beck who kindly
restored~\cite{BP03b} for us.  This project was initiated when
the authors were hosted at the Mittag-Leffler Institute in
the (rather eventful) Spring of~2020; we are grateful for
the hospitality.  The first author was partially supported by the
Knut and Alice Wallenberg Foundation and the G\"oran Gustafsson
Foundation. The third author was partially supported by the NSF.
}

\vskip.6cm


\end{document}